\documentclass[reqno,oneside, 11pt,letterpaper]{amsart}

\usepackage[usenames,dvipsnames]{xcolor}
\usepackage[colorlinks=true,linkcolor=Blue,citecolor=teal]{hyperref}
\usepackage{amsmath,amssymb,epsf}
\usepackage{amsfonts,amsbsy,amscd,stmaryrd}
\usepackage{latexsym,amsopn}
\usepackage{amsthm,mathtools}
\usepackage{mathrsfs}
\usepackage{slashed}
\allowdisplaybreaks

\usepackage{graphicx}

\usepackage{color}
\definecolor{Red}{rgb}{1.,0.,0.}
\newcounter{smallarabics}
\newenvironment{arabicenumerate}
{\begin{list}{{\normalfont\textrm{(\arabic{smallarabics})}}}
  {\usecounter{smallarabics}\setlength{\itemindent}{0cm}
   \setlength{\leftmargin}{5ex}\setlength{\labelwidth}{4ex}
   \setlength{\topsep}{0.75\parsep}\setlength{\partopsep}{0ex}
   \setlength{\itemsep}{0ex}}}
{\end{list}}

\newcounter{smallroman}

\newcommand{\ben}{\begin{arabicenumerate}}
\newcommand{\een}{\end{arabicenumerate}}

\def\init{\setcounter{equation}{0}}

\newtheorem{theorem}{Theorem}[section]
\newtheorem{assumption}{Hypothesis}[section]
\newtheorem{proposition}[theorem]{Proposition}
\newtheorem{lemma}[theorem]{Lemma}
\newtheorem{corollary}[theorem]{Corollary}
\theoremstyle{definition}
\newtheorem{definition}[theorem]{Definition}
\newtheorem{remark}[theorem]{Remark}
\newtheorem{example}[theorem]{Example}
\newcommand{\beq}{\begin{equation}}
\newcommand{\eeq}{\end{equation}}

\newcommand{\bea}{\begin{aligned}}
\newcommand{\eea}{\end{aligned}}

\newcommand{\bex}{\begin{example}}
\newcommand{\eex}{\end{example}}
\def\bel{\begin{lemma}}
\def\eel{\end{lemma}}
\def\bet{\begin{theorem}}
\def\eet{\end{theorem}}
\def\bed{\begin{definition}}
\def\eed{\end{definition}}
\def\ber{\begin{remark}}
\def\eer{\end{remark}}

\renewcommand{\sec}[1]{\S\ref{#1}}
\newcommand{\secs}[2]{\S\S\ref{#1}--\ref{#2}}
\renewcommand{\leq}{\leqslant}
\renewcommand{\geq}{\geqslant}

\newcommand{\step}[1]{{\noindent\emph{Step #1.\,}}}

\def\rr{{\mathbb R}}

\def\cc{{\mathbb C}}
\def\nn{{\mathbb N}}

\def\part{{\rm par}}

\def\slim{{\rm s-}\lim}

\def\cinf{C^\infty}

\def\APS{{\rm APS}}

\usepackage{calrsfs}
\DeclareMathAlphabet{\pazocal}{OMS}{zplm}{m}{n}

\def\cY{{\pazocal Y}}

\def\cR{{\pazocal R}}

\def\cN{{\pazocal N}}

\def\wf{{\rm WF}}
\def\mod{{\rm \ \ mod \ }}

\def\ad{{\rm ad}}

\def\loc{{\rm loc}}

\def\sfl{{\rm sf}}

\let\Im\relax
\let\Re\relax
\let\div\relax
\let\sf\relax
\DeclareMathOperator{\Ker}{Ker}
\DeclareMathOperator{\Ran}{Ran}
\DeclareMathOperator{\Im}{Im}
\DeclareMathOperator{\Re}{Re}

\DeclareMathOperator{\ind}{ind}
\DeclareMathOperator{\div}{div}
\DeclareMathOperator{\spe}{sp}
\DeclareMathOperator{\sf}{sf}
\DeclareMathOperator{\End}{End}

\renewcommand\ker{{\rm ker}}

\newcommand{\qeds}{\qed\medskip}

\def \p{ \partial}

\def\12{\frac{1}{2}}
\def\14{\frac{1}{4}}

\DeclareMathOperator{\supp}{supp}

\newcommand{\one}{\boldsymbol{1}}

\def\cH{{\pazocal H}}

\def\c{{\rm c}}
\def\F{\module{H}^{\inv}}

\def\cF{{\pazocal F}}

\def\cX{{\pazocal X}}
\def\cK{{\pazocal K}}

\def\12{\frac{1}{2}}

\def\ad{{\rm ad}}

\def\bep{\begin{proposition}}
\def\eep{\end{proposition}}

\def\cE{\pazocal{E}}

\newcommand{\bra}{\langle}
\newcommand{\ket}{\rangle}

\def\init{\setcounter{equation}{0}}

\DeclareSymbolFont{boldoperators}{OT1}{cmr}{bx}{n}
\SetSymbolFont{boldoperators}{bold}{OT1}{cmr}{bx}{n}

\makeatletter

\newcommand*{\defeq}{:=}				
\newcommand*{\eqdef}{=:}							
\makeatother

\def\Op{{\rm Op}}

\newcommand{\Red}{}
\def\fantom{\\ &\phantom{=}\,}
\def\cf{C^\infty}
\def\zero{{\rm\textit{o}}}

\def\c{{\rm c}}

\newcommand{\cHH}{\cH^+_-\oplus \cH_+^-}

\def\loc{{\rm loc}}

\let\origmaketitle\maketitle
\def\maketitle{
  \begingroup
  \def\uppercasenonmath##1{} 
  \let\MakeUppercase\relax 
	\origmaketitle
  \endgroup
}

\def\beproof{
\noindent{\bf Proof.}\ \ }
\renewenvironment{proof}{\beproof}{\qeds}
\newenvironment{refproof}[1]{\smallskip
\noindent{\bf Proof of {#1}.}\ \ }{\qeds}

\newcommand{\inv}{{\scriptscriptstyle (-1)}}

\def\spexi{\xi}
\def\rx{x}

\newcommand{\open}[1]{\mathopen{}\mathclose{\left]#1 \right[}}
\newcommand{\clopen}[1]{\mathopen{}\mathclose{\left[#1 \right[}}
\newcommand{\opencl}[1]{\mathopen{}\mathclose{\left]#1 \right]}}

\newcommand{\Eucl}{{\rm\scriptscriptstyle E}}

\def\dvol{\mathop{}\!d{\rm vol}}

\def\SM{SM}

\def\SpM{S^+\!M}

\def\dual{\!\cdot \!}
\def\st{{ \ |\  }}
\newcommand{\norm}[1]{\left\|{#1}\right\|}
\newcommand{\module}[1]{\left|#1\right|}

\def\td{{\rm td}}
\newcommand{\pdo}[3]{\Psi^{#1,#2}_{\td}(#3;E)}
\newcommand{\pdos}[3]{\Psi^{#1,#2}_{\td}(#3;S\Sigma)}
\newcommand{\sym}[3]{S^{#1,#2}_{\td}(#3;E)}

\newcommand{\sob}[1]{H^{#1}(\Sigma;E)}
\newcommand{\B}[2]{B(\sob{#1},\sob{#2})}

\newcommand{\Aform}{\widehat{\rm A}}

\renewcommand{\iff}{\xLeftrightarrow{\phantom{--}}}

\makeatletter
\newsavebox\myboxA
\newsavebox\myboxB
\newlength\mylenA

\newcommand*\xoverline[2][0.75]{%
    \sbox{\myboxA}{$\m@th#2$}%
    \setbox\myboxB\null
    \ht\myboxB=\ht\myboxA%
    \dp\myboxB=\dp\myboxA%
    \wd\myboxB=#1\wd\myboxA
    \sbox\myboxB{$\m@th\overline{\copy\myboxB}$}
    \setlength\mylenA{\the\wd\myboxA}
    \addtolength\mylenA{-\the\wd\myboxB}%
    \ifdim\wd\myboxB<\wd\myboxA%
       \rlap{\hskip 0.8\mylenA\usebox\myboxB}{\usebox\myboxA}%
    \else
        \hskip -0.9\mylenA\rlap{\usebox\myboxA}{\hskip 0.9\mylenA\usebox\myboxB}%
    \fi}
\makeatother

\makeatletter
\newsavebox\xmyboxA
\newsavebox\xmyboxB
\newlength\xmylenA

\newcommand*\xxoverline[2][0.75]{%
    \sbox{\xmyboxA}{$\scriptstyle\m@th#2$}%
    \setbox\xmyboxB\null
    \ht\xmyboxB=\ht\xmyboxA%
    \dp\xmyboxB=\dp\xmyboxA%
    \wd\xmyboxB=#1\wd\xmyboxA
    \sbox\xmyboxB{$\scriptstyle\m@th\overline{\copy\xmyboxB}$}
    \setlength\xmylenA{\the\wd\xmyboxA}
    \addtolength\xmylenA{-\the\wd\xmyboxB}%
    \ifdim\wd\xmyboxB<\wd\xmyboxA%
       \rlap{\hskip 0.8\xmylenA\usebox\xmyboxB}{\usebox\xmyboxA}%
    \else
        \hskip -0.9\xmylenA\rlap{\usebox\xmyboxA}{\hskip 0.9\xmylenA\usebox\xmyboxB}%
    \fi}
\makeatother

\def\M{\xoverline{M}}
\def\Ms{{\xxoverline{M}}}
\def\pM{\partial \M}
\def\pMs{{\partial\Ms}}

\def\cp{T} 
\def\cm{T^{-1}}

\author{}
\address{Laboratoire Jacques-Louis Lions, Sorbonne Universit\'e, 75252 Paris, France}
\email{shend@ljll.math.upmc.fr}
\author[]{\normalsize Dawei \textsc{Shen} \& Micha{\l} \textsc{Wrochna}}
\address{CY Cergy Paris Universit\'e, 2 av.~Adolphe Chauvin, 95302 Cergy-Pontoise, France}
\email{michal.wrochna@cyu.fr}

\begin{document}

\title[]{\Large An   index theorem  on asymptotically static spacetimes\\ with compact Cauchy surface}

\begin{abstract} We consider the Dirac operator on asymptotically static  Lorentzian manifolds with an odd-dimensional compact Cauchy surface. We prove that if Atiyah--Patodi--Singer boundary conditions are imposed at infinite times then the Dirac operator is Fredholm.  This generalizes  a theorem due to Bär--Strohmaier \cite{BS} in the case of finite times, and we also show that the corresponding index formula extends to the infinite  setting.
Furthermore, we demonstrate the existence of a Fredholm inverse which is at the same time a Feynman parametrix in the sense of Duistermaat--Hörmander. The proof combines methods from time-dependent scattering theory with a variant of Egorov's theorem for pseudo-differential hyperbolic systems.
\end{abstract}

\maketitle

\section{Introduction}

\subsection{Introduction and main result} In the last few years it was realized that hyperbolic operators on Lorentzian manifolds have an interesting  Fredholm and spectral theory despite their non-ellipticity and non-hypoellipticity  \cite{Vasy2013,BVW,GHV,BS,derezinskisiemssen0,massivefeynman1,vasyessential}. Furthermore, it turned out that it is possible to show relationships with geometry that have surprisingly many common features with the Riemannian setting \cite{BS,Strohmaier2020b,Braverman2020,Dang2020a,Baer2020a}.



 In particular, Bär--Strohmaier \cite{BS}  proved an index formula  which can be interpreted as the Lorentzian version of the  Atiyah--Patodi--Singer theorem for the Dirac operator on compact spacetimes with space-like boundary.  A generalisation to spatially non-compact manifolds was then shown by Braverman \cite{Braverman2020}   in the case  of a confining Callias potential, and different boundary conditions were studied by Bär--Hannes \cite{Bar2018}.

On the other hand, relativistic physics  has provided strong motivation for considering hyperbolic operators on Lorentzian manifolds that are \emph{non-compact in temporal directions}.   This has included proofs of the Fredholm property or global invertibility of the wave and Klein--Gordon operator by Vasy \cite{Vasy2013} and Hintz \cite{Hintz2017} on asymptotically de Sitter and Kerr--de Sitter spacetimes and by Baskin--Vasy--Wunsch \cite{BVW},  Gell-Redman--Haber--Vasy \cite{GHV}, Gérard--Wrochna \cite{massivefeynman1,Gerard2018,Gerard2020} and Taira \cite{Taira2020a} for asymptotically Minkowski spacetimes, cf.~Derezi\'nski--Siemssen \cite{derezinskisiemssen0,derezinskisiemssen2}, Vasy \cite{vasyessential}  and Nakamura--Taira \cite{nakamurataira} for the closely related question of essential self-adjointness. However, as these results do not concern the Dirac operator, they are unlikely to be tied to a rich index theory, and typically the index is zero indeed.

The objective of the  present work  is to prove an index theorem for the Dirac operator on a class of Lorentzian manifolds that are non-compact in time, generalizing in this way the result of Bär--Strohmaier, and to analyze  microlocal properties of parametrices needed to make the connection with the local geometry. 

Namely, we make the following assumption on the geometry. Suppose that $(M,g)$ is an oriented and time-oriented smooth Lorentzian manifold such that $M=\rr\times \Sigma$ and
\beq
\Sigma \mbox{ is compact and odd-dimensional}.
\eeq
We assume that for each $t\in\rr$, $\{t\}\times \Sigma$ is a Cauchy surface. Thus, $(M,g)$ is a \emph{globally hyperbolic spacetime} with compact Cauchy surface.

\begin{assumption}\label{hypothesis1}  We assume the metric  to be of the form
\beq
g=- c^2(t)dt^2 + h(t),
\eeq
where $c\in \cf(M)$ satisfies $c>0$ and  $\rr\ni t\mapsto h(t)\in \cf(\Sigma;T^*\Sigma\otimes_{\rm s}T^*\Sigma)$ is a smooth  family of Riemannian metrics on $\Sigma$ such that for some $\delta>1$,
\begin{enumerate}
\item\label{thyp1} $c(t)-c_\pm\in S^{-\delta}(\rr_\pm,\cf(\Sigma))$ for some $c_\pm\in \cf(\Sigma)$ s.t.~$c_\pm>0$,
\item\label{thyp2}  $h(t)-h_\pm\in S^{-\delta}(\rr_\pm,\cf(\Sigma;T^*\Sigma\otimes_{\rm s}T^*\Sigma))$ for some Riemannian metric $h_\pm$.
\end{enumerate}
\end{assumption}

This means in particular that $g$ decays as $t\to\pm\infty$ to a static (product type) metric at short-range rate $\module{t}^{-\delta}$, and differentiating the coefficients $k$ times in $t$ yields decay of order $ \module{t}^{-\delta-k}$. Thus, $g$ is an \emph{asymptotically static metric}. Through the change of coordinates $\tau_\pm=e^{\mp t}$ for $t$ close to $\pm\infty$ we can  compactify $M$, which results in a manifold $\M$ with boundary $\pM=\{ \tau_+=0\}\cup \{\tau_- =0\}$ with two components representing  future and past infinity  $t=\pm\infty$. 

\begin{remark} Close to the boundary, if we write $\tau=\tau_\pm$,  then $dt^2=\pm  \frac{d\tau^2}{\tau^2}$.  We remark that in  Hypothesis \ref{hypothesis1}  we do not assume that $c_{|\pMs}=c_\pm=1$. Furthermore,  while the conditions \eqref{thyp1}--\eqref{thyp2} imply \emph{conormality} of the metric coefficients, they do not  imply smoothness in $\tau$ at the boundary. This is the main difference as compared to (the Lorentzian analogue of) the  class of \emph{exact b-metrics} considered by Melrose \cite{Melrose1993} in his version of the Atiyah--Patodi--Singer theorem. 
\end{remark}

We consider the Dirac operator $D$ acting on positive-chirality sections of a spinor bundle $SM$ over $M$, see \sec{sec2} for details on the terminology. At each time $t\in \rr$, $D$ induces a Riemannian Dirac operator denoted by $A(t)$, and there are two boundary Dirac operators $A_\pm$ corresponding to the limits $t\to\pm\infty$. If $I\subset\rr$ is an interval we denote by $\one_I(A(t))$ the corresponding spectral projection. 

We consider $D$ as an operator   $D:\cX \to \cY$, where $\cY$ is a weighted $L^2$-space with a sufficiently decaying weight in $t$ (the precise choice  has no importance for the result), and  $\cX$ is the completion of smooth spinors in the corresponding $L^2$ graph norm, see \sec{ss:ind}.  We then define $D_{\rm APS}$ as the restriction of $D$ to the space $\cX_{\APS}$ consisting of all $u\in \cX$ such that
\beq\label{eq:APSc}
\lim_{t\to+\infty} \one_{\opencl{-\infty,0}}(A(t))u(t)=0, \quad \lim_{t\to-\infty} \one_{\clopen{0,+\infty}}(A(t))u(t)=0. 
\eeq
As in the compact setting of \cite{BS}, \eqref{eq:APSc} are interpreted as Atiyah--Patodi--Singer boundary conditions at  $\pM$. Our first result is the following theorem. 

\begin{theorem}[{cf.~Theorem \ref{thm:final}}] \label{thm:main} The  operator $D_{\rm APS}: \cX_{\rm APS} \to \cY$ is Fredholm of index
$$
\ind(D_{\rm APS})=\int_M \widehat{\rm A}+ \int_{\pMs} {\mathrm T}\widehat{\rm A}+\12\big(\eta(A_+,A_-) -\dim\ker(A_+)-\dim\ker(A_-)\big),
 $$
 where $\Aform$ the Atiyah--Singer integrand (or  $\Aform$-form) on $(M,g)$,  ${\mathrm T}\widehat{\rm A}$  is the  transgression form, and $\eta(A_+,A_-)=\eta(A_+)-\eta(A_-)$ is the difference of the eta forms of $A_+$ and $A_-$.
\end{theorem}

The content of the above index formula is exactly analogous to the result of Bär--Strohmaier (and its generalisation \cite{Braverman2020}, where $\eta(A_+,A_-)$ is defined as a relative eta invariant \cite{Braverman2021}), which in turn closely parallels  the Atiyah--Patodi--Singer theorem in Riemannian signature, see \cite{Atiyah1975a,Gerd1992,Gilkey1993,Melrose1993}. 

On the other hand, a significant difference with the results known in the finite-time case can be seen on the level of local properties of parametrices of $D_{\rm APS}$. 

The essential question is how do parametrices of $D_{\rm APS}$ fit into the Duistermaat--Hörmander theory of Fourier integral operators \cite{DH}, or more precisely, does  $D_{\rm APS}$ have parametrices with \emph{Feynman wavefront set}, a microlocal condition (see Definition \ref{deffeynman}) which characterizes a parametrix uniquely modulo smoothing operators. In the finite-time setting of \cite{BS} this property was only shown to hold true for metrics that are exactly static (product-type) in a neighborhood of the boundary, yet the results for wave operators on non-compact spacetimes  \cite{GHV,vasywrochna,massivefeynman1,vasyessential} suggest that this assumption might not necessarily be needed in situations like Hypothesis \ref{hypothesis1}. Our second main result states that the Feynman property holds true indeed.

\begin{theorem}[{cf.~Theorem \ref{thm:final2}}]\label{thm:maint2} $D_\APS$ admits a parametrix which has Feynman wavefront set.
\end{theorem}

The importance of the Feynman property stems primarily from applications in Quantum Field Theory on curved spacetimes, as already outlined in the work of Duistermaat--Hörmander \cite{DH} and further developed by many other authors. Very recently, however, it was  demonstrated that is  also essential in relating the global theory of hyperbolic operators with the local geometry (as it enables the use of the Hadamard parametrix as an intermediary). In particular, it plays a central rôle  in the proof of the Lorentzian spectral action expansion and Kastler--Kalau--Walze theorem by Dang--Wrochna \cite{Dang2020a}.  In the Dirac case, the Feynman wavefront set is the key assumption in the approach of Bär--Strohmaier \cite{Baer2020a} to local index theory in the Lorentzian signature.  Thus, as a consequence of Theorem \ref{thm:maint2}, the local index theorem  \cite{Baer2020a} applies to the geometric setting of Hypothesis \ref{hypothesis1}.

\subsection{Plan of the proofs} The proof of the index formula in \cite{BS} is based on a spectral flow argument and a reduction to Euclidean signature: this turns  out to be very robust and can be applied to our setting with only a few necessary adaptations (see also \cite{ronge,dun3} for a more abstract version). On the other hand, it is the proof of the Fredholm property in Theorem \ref{thm:main} which requires the introduction of new techniques {\Red(see \cite[\S6]{ronge} for a related  example where the Fredholm property is not true)}. The method used in \cite{BS} {\Red(see also \cite{Damaschke2021} for a generalization to Galois coverings)} is based on the calculus of Fourier integral operators and does not apply to the present situation. Instead, we use methods from time-dependent scattering theory combined with a long-time variant of Egorov's theorem for hyperbolic systems. The arguments are then refined to prove the wavefront set estimate stated in Theorem \ref{thm:maint2}. More precisely, the proofs are structured as follows.\smallskip

\ben
\item\label{istep1} In \sec{sec2} (see also Lemma \ref{lemrel}) we reduce the Dirac operator to the evolutionary form
$\p_t - i H(t)$, where $H(t)$ is a differential operator which is elliptic and similar to a self-adjoint operator  for each $t\in \rr$.   \smallskip

\item\label{istep2} We consider the Schr\"odinger propagator $U(t,s)$ generated by $H(t)$ and analyze in \sec{sec:scattering} the large-time Heisenberg evolution of spectral projections. The key result is Proposition \ref{lemkey}, which states that for suitable smooth functions $\chi$, the operator 
$$U(t,0)  \chi (H(0)) U(0,t) - \chi(H(t))$$
is  compact, and it acquires extra  smoothing properties in the limit  $t\to \pm\infty$. For all $t$ it can also be expressed in terms of pseudo-differential operators of order $-1$, and in this sense Proposition \ref{lemkey} is a variant of Egorov's theorem. The proof uses a time-dependent pseudo-differential calculus developed by Gérard--Wrochna \cite{inout} and adapted to the present setting in \sec{sec:pdo} including a Beals type characterization.\smallskip

\item\label{istep3} In \sec{sec:fredholm}, the Atiyah--Patodi--Singer boundary conditions are interpreted in terms of asymptotic data in the sense of time-dependent scattering theory. More precisely, as comparison dynamics we take $e^{ it H(t)}$, we show in \sec{scatteringsection}      that $W(0,t)\defeq U(0,t)e^{i t H(t)}$   is well-behaved as $t\to \pm \infty$, and we use this in \sec{ss:fp} to give meaning to asymptotic data of solutions. We remark that this choice of comparison dynamics is unusual from the point of view of scattering theory but serves our purpose particularly well.\smallskip

\item\label{istep4} The Fredholm property is concluded in \secs{sec:fredholm}{ss:index}    from  \eqref{istep1}--\eqref{istep3} in combination with abstract Fredholm theory arguments in the spirit of \cite{BS}, and then the index formula follows in a similar way as in \cite{BS}.\smallskip

\item\label{istep5} The proof of the Feynman property in \sec{sec:feynman}  is essentially reduced  to showing that if $u$ satisfies $Du\in \cf$ and
$$
\lim_{t\to+\infty} \one_{\clopen{0,+\infty}}(A(t))u(t)=0,
$$
then its wavefront set $\wf(u)$ is contained in at most one of the two components $\cN^\pm$ of the characteristic set of $D$. This statement is shown to be a consequence of \eqref{istep2} (this contrasts with \cite{BS} where  a parametrix gluing contruction is used in the finite-time case). The reduction uses in a crucial way positivity properties of Fredholm inverses of $D_{\rm APS}$ which are proved in \sec{ss:positivity}.\smallskip
\een

We remark that the sections \secs{sec:scattering}{sec:pdo} involving the key technical result, Proposition \ref{lemkey},  are self-contained and  can be of independent interest.

\subsection{Bibliographical remarks} A primary application of index formulae is the computation of chiral anomalies in Quantum Field Theory, see Bär--Strohmaier \cite{BS2}, cf.~Zahn \cite{Zahn2015} and  Bär--Strohmaier \cite{Baer2020a} for local versions. 

Index formulae for the scattering matrix of the Dirac operator were previously obtained by Matsui \cite{Matsui1987,Matsui1990}, Bunke--Hirschman \cite{Bunke1992} and Pankrashkin--Richard \cite{Pankrashkin2014} (cf.~Finster \cite{Finster2017} for the index  of a Lorentzian signature operator). We point out that in our setting $\ind(D_{\rm APS})$ equals the index of a scattering matrix component (see Proposition \ref{indexflow} and the proof of Theorem \ref{thm:final}) for a non-standard comparison dynamics, and one can relate it to the usual scattering matrix under some additional assumptions.

Generalisations and variants of Egorov's theorem for pseudo-differential systems   were considered among others by Cordes \cite{Cordes1982,cordes},  Jakobson--Strohmaier \cite{Jakobson2007} and  Kordyukov \cite{Kordyukov2007}, cf.~Brum\-melhuis--Nourrigat \cite{Brummelhuis1999}, Bolte--Glaser \cite{Bolte2004} and Assal \cite{Assal2016} for the semi-classical case; see also   Capoferri--Vassiliev \cite{Capoferri2021} for the closely related construction of pseudo-differential projections. Our proof of  Proposition \ref{lemkey} is most closely related to \cite{Cordes1982} {\Red and to recent works by Gérard--Stoskopf \cite{GS,Gerard2021a} inspired by time-dependent projections in adiabatic theory, see e.g.~Sjöstrand \cite{Sjostrand1993}}.

The wavefront set estimate for parametrices of $D_{\rm APS}$ is closely related to the construction of Hadamard states from asymptotic data in Quantum Field Theory, a problem considered in different settings by authors including Moretti \cite{Moretti2008}, Dappiaggi--Pinamonti--Mo\-ret\-ti \cite{Dappiaggi2009}, G\'erard--Wrochna \cite{inout,massivefeynman1}, Vasy--Wrochna \cite{vasywrochna} {\Red and Gérard--Stoskopf \cite{Gerard2021a}}. Positivity properties are crucial  in this context, and have also been studied for  Feynman parametrices in the already mentioned references \cite{DH,massivefeynman1,BS} and in works by Vasy \cite{Vasy2017b} and Islam--Strohmaier \cite{Islam2020}. The time-dependent pseudo-differential calculus and approximate diagonalisation of the evolution used in \cite{massivefeynman1,Gerard2021a} are particularly relevant to the present setting, though the approach here is ultimately different (furthermore, in contrast to \cite{massivefeynman1,Gerard2021a} we do not assume $c_\pm=1$, which leads to new complications). 

On the side note we remark that if the boundary is assumed to be time-like instead of being space-like, the Lorentzian Dirac operator behaves very differently,  see Drago--Gro{\ss}e--Murro for the well-posedness of the corresponding Cauchy problem \cite{Drago2021}

Finally, as our result assumes compactness of $M$ in spatial directions,  one might ask if an index formula could be shown for e.g.~perturbations of Minkowski space and other classes of non-compact spacetimes. We hope  that the  advances  including e.g.~\cite{BVW, GHV,massivefeynman1,Braverman2021,Braverman2020,Baer2020a,Dang2020a,Stoskopf} make it a viable goal for future research. Generalizations in the spirit of the works of Bismut--Cheeger \cite{Bismut1990,Bismut1990a} or Melrose--Piazza \cite{Melrose1997} also remain an open problem.

\subsection*{Acknowledgments} The authors are particularly grateful to Dean Baskin, Christian Bär, Nguyen Viet Dang,  Penelope Gehring, Jesse Gell-Redman, Christian Gérard and Théo Stoskopf for helpful discussions. Support from the grants ANR-16-CE40-0012-01, ANR-20-CE40-0018 is gratefully acknowledged. M.W.~thanks  the MSRI in Berkeley and the Mittag--Leffler Institute in Djursholm for the kind hospitality during thematic programs and workshops in 2019--20.

\section{The Dirac operator in Lorentzian signature}\label{sec2}

\subsection{Geometric setup}\label{ss1} Let $(M,g)$ be a $1+d$-dimensional globally hyperbolic spacetime\footnote{Recall that an oriented and time-oriented smooth Lorentzian manifold $(M,g)$  is called a \emph{globally hyperbolic spacetime} if it has a \emph{Cauchy surface}, i.e.,  a smooth hypersurface which is intersected by every maximally extended, non-spacelike curve exactly once.} with compact Cauchy surface $\Sigma$. We assume that the dimension $d$ of $\Sigma$ is \emph{odd}.

Let us recall the necessary background on spinors and Dirac operators, see e.g.~\cite{Baum1981,Lawson1989,BGM,Trautman2008} for more detailed accounts.

Suppose $\SM\to M$ is a complex spinor bundle. We denote by $\cf(M;\SM)$ the space of its smooth sections, and use an analogous notation for sections of other vector bundles.

Let us recall that the spinor bundle $\SM$ is endowed with a linear map
$$
\gamma: \cinf(M; TM)\to \cinf(M; \End(\SM))
$$ called \emph{Clifford multiplication} and satisfying
\beq\label{clifford}
\gamma(X)\gamma(Y)+ \gamma(Y)\gamma(X)= -2( X\dual gY) \one, \ \ X, Y\in \cinf(M; TM).
\eeq
Furthermore, one is given a connection $\nabla^{SM}$ on $SM$, called \emph{spin connection}, such that in particular, for all $X, Y\in \cinf(M; TM)$ and  $\psi\in \cinf(M;SM)$,
$$
\nabla_{X}^{\SM}(\gamma(Y)\psi)= \gamma(\nabla_{X}Y)\psi+ \gamma(Y)\nabla_{X}^{\SM}\psi,
$$
where $\nabla$ is the Levi-Civita connection on $(M,g)$.
In the physicist's terminology, the \emph{massless Dirac operator} is the differential operator $\slashed{D}$ (or $i$ times $\slashed{D}$) given in a time-oriented local frame  $(e_{0}, e_1,\dots, e_{d})$ of $TM$ by
$$
\slashed{D}=g^{\mu\nu}\gamma(e_{\mu})\nabla^{SM}_{e_{\nu}} : \cf(M;SM)\to  \cf(M;SM)
$$
using Einstein's summation convention. The section $\Gamma=i^{d(d+3)/2}\gamma(e_0)\cdots\gamma(e_d)\in \cf(M;\End(SM))$ satisfies $\Gamma^2=\one$. The spinor bundle has therefore a decomposition $SM=S^+M\oplus S^- M$, where $S^\pm M$ is fiberwise the eigenspace of $\Gamma$ for the eigenvalue $\pm 1$.  Recall that we have   assumed that $d$ is odd, and consequently one deduces from \eqref{clifford} that $\Gamma$ anti-commutes with all $\gamma(X)$, hence $\slashed{D}\Gamma=-\Gamma\slashed{D}$. Thus, in terms of the $S^+M\oplus S^- M$ decomposition,
$$
\slashed{D}=\begin{pmatrix*}[c] 0 & D^- \\ {D}^+ & 0 \end{pmatrix*},
$$
where $$
D^\pm: \cf(M; S^\pm M)\to \cf(M;S^\mp M).
$$
The differential geometry literature often calls  $D^+$  the {Dirac operator} (at slight risk of confusion with $\slashed{D}$).  Note that in the physics literature the less ambiguous name \emph{Weyl operator} is usually used for $iD^+$.

\subsection{Foliation by Cauchy surfaces}\label{ss:foliation} By global hyperbolicity {\Red and the Bernal--S\'anchez theorem \cite{bernal1,bernal2}}, $M\cong\rr\times \Sigma$ and there exists a foliation $\{ \Sigma_t\}_{t\in \rr}$ by Cauchy hypersurfaces with $\Sigma_t=\{t\}\times \Sigma$.  Furthermore, there exists $c\in \cf(M)$ strictly positive and  a smooth family of Riemannian metrics $\rr\ni t\mapsto h(t)$  on $\Sigma$, such that (disregarding the diffeomorphism $M\to \rr\times \Sigma$ in the notation)
\beq
g=- c^2(t)dt^2 + h_{ij}(t)dy^i dy^j,
\eeq
where we sum over repeated indices, and the dependence on the spatial variable $y$ is dropped in the notation. We denote $|h(t)|\defeq\left|\det h(t) \right|$.

Let $n\in \cf(M;TM)$ be the unique past-directed vector field such that
$$
n(x)\dual g(x) n(x)=-1 \mbox{ and } n(x) \dual g(x) v=0
$$
for all $x=(t,y)\in M$ and $v\in T_y\Sigma_t$. Then, $\beta\defeq \gamma(n)\in \cf(M;\End(\SM))$ satisfies $\beta^2=\one$ and
\beq\label{eq:scpr}
\forall u\in\cf(M;SM), \ \bra u | u  \ket \defeq ( u | \beta u )_{\SM}\geq 0.
\eeq
The associated Hilbert space will be denoted by $L^2(M;SM)$, and $L^2(M;S^+M)$ is defined analogously.

{\Red By the results in \cite{BGM} and using the assumption that $\Sigma$ is odd-dimensional,   for any $t\in\rr$ the  restriction $(S^\pm M)|_{\Sigma_t}$ can be identified with  a spinor bundle   $S\Sigma_t$ on $\Sigma_t$.  It is equipped with the  Clifford multiplication $\gamma_t \in \cf(\Sigma_t;\End(S\Sigma_t))$ defined by restricting $i\beta \gamma$ to $\Sigma_t$. 
 For  any $t,s\in\rr$ let
\beq
\tau_t^s : (SM|_{\Sigma_t}) \to SM|_{\Sigma_t}
\eeq
be for each base point $y\in\Sigma$ the parallel transport for the spin connection $\nabla^{SM}$ along the curve $\rr\ni t\mapsto(t,y)\in M$. We denote $\tau_t\defeq\tau_t^0$ to simplify the notation. 
In what follows we will often write  $S\Sigma$  instead of $(S^+M|_{\Sigma_0})$  for the sake of brevity. 

For each $t\in\rr$, let $L^2(\Sigma_t; S^+M |_{\Sigma_t})$ be the Hilbert space defined similarly as $\bra \cdot | \cdot \ket$ using $\beta$ and the volume form $\dvol_{h(t)}$. Let $\rho(t)$ be the unique function such that $\dvol_{h(t)}=\rho(t)^2 \dvol_{h(0)}$, or more explicitly, $\rho(t)=|h(0)|^{-\frac{1}{4}}|h(t)|^{\frac{1}{4}}$. Then, the map
$$
U(t)\defeq \rho(t) \tau_t  : L^2(\Sigma_t; S^+M|_{\Sigma_t})\to  L^2(\Sigma; S\Sigma)
$$
is  invertible. }

\subsection{Dirac operators on Cauchy surfaces}\label{ss3} For each $t\in\rr$, the spinor connection $\nabla^{\SM}$ induces a spinor connection on $S\Sigma_t$, denoted by $\nabla^{S\Sigma_t}$, and satisfying (see \cite[(3.5)]{BGM})
$$
\nabla^{\SM}_X = \nabla^{S\Sigma_t}_X - \12 \beta \gamma(\nabla_{\!X}^{\phantom n} n), \ \ X\in \cf(\Sigma_t; T\Sigma_t).
$$
For $t\in \rr$, let $A(t)$ be the Dirac operator associated to the spinor connection $\nabla^{S\Sigma_t}$ and the Clifford multiplication $\gamma_t$, i.e.
$$
A(t) = h^{ij}(t)\gamma_t(e_{i})\nabla^{S\Sigma_t}_{e_{j}} : \cf(\Sigma_t;SM|_{\Sigma_t})\to  \cf(\Sigma_t;SM|_{\Sigma_t}).
$$
The Dirac operator $A(t)$ is elliptic and formally self-adjoint in $L^2(\Sigma_t; S^+M|_{\Sigma_t})$. Its closure, also denoted $A(t)$, has discrete spectrum.

The Dirac operator $\slashed{D}$ on $\SM$ can be expressed as follows in terms of $A(t)$ (see \cite[(3.6)]{BGM}):
\beq
\slashed{D}= \beta \begin{pmatrix}  -\nabla_n^{\SM}-i {A}(t)-r(t) & 0 \\ 0 &  -\nabla_n^{\SM}+i {A}(t)-r(t)  \end{pmatrix},
\eeq
where $r(t)$ is the multiplication operator by $\frac{d}{2}$ times the mean curvature of $\Sigma_t$. Our main object of interest is the  operator
\beq\label{eq:defD1}
D= -\nabla_n^{\SM}-i {A}(t)-r(t) :  \cf(M;S^+ M)\to\cf(M; S^+M).
\eeq
The practical significance of considering $D$ instead of $D^+$ is that the former acts on sections of the same bundle. Note that $D$ and $D^+$ are simply related through the isomorphism $\beta$, so these two choices are equivalent  from the point of view of index theory. By  abuse of terminology we also occasionally refer to $D$ as the Dirac operator.

The family  $\{ L^2(\Sigma_t; S^+M|_{\Sigma_t})\}_{t\in \rr}$ and more generally, the family  $\{ H^m(\Sigma_t;S^+M|_{\Sigma_t})\}_{t\in \rr}$ of Sobolev spaces for $m\in\rr$, can be considered as a bundle of Hilbert spaces over $\rr$, trivialized by the parallel transport $\tau_t$.  Let $C^0_t L^2_y(M;\SpM)$, and more generally $C^0_t H^m_y(M;\SpM)$ be the space of continuous sections of that bundle. Seminorms of $u\in C^0_t H^m_y(M;\SpM)$ are by  definition the $C^0$ seminorms of 
$$
\rr\ni t\mapsto \| u(t) \|_{H^m(\Sigma_t; S^+M|_{\Sigma_t})}.
$$
  Furthermore, if we set
\beq\label{eq:convention}
\big(U u)(t)\defeq U(t) u(t)
\eeq
then
$$
U  :   C^0_t L^2_y(M;\SpM)\to C^0(\rr, L^2(\Sigma;S\Sigma))
$$
is an isomorphism. Next, we define $L^2_t H^m_y(M;\SpM)$ in a similar way as the space of weighted $L^2$ sections of the bundle $\{ H^m(\Sigma_t; S^+M|_{\Sigma_t})\}_{t\in \rr}$ with  norm given by the  $L^2$-norm of  $\rr\ni t\mapsto \| u(t) \|_{H^m(\Sigma_t; S^+M|_{\Sigma_t})}$ associated with the density $c(t) dt$. Then, $U$ extends  to an isomorphism
$$
 U :   L^2_t H^m_y (M;\SpM)\to L^2(\rr, H^m(\Sigma;S\Sigma)).
$$
Thanks to the presence of the weight $c(t)$ in the definition,  the space $L^2_t L^2_y (M;\SpM) =L^2_t H^0_y (M;\SpM)$ can be identified with $L^2(M;\SpM)$.


\subsection{Reduction to an evolutionary equation}

The next lemma (which follows from the computations in \cite{dun1}), reduces $D$ to an evolutionary form which is particularly useful for us.

\begin{lemma}\label{lem:equivalence}
The operator $D$ defined in \eqref{eq:defD1} satisfies
\beq\label{eq:equivalence}
D= U(t)^{-1}c^{-1}(t) \big(\p_t - i H(t)\big) U(t),
\eeq
 where  $H(t)=  c(t) U(t) A(t) {U(t)}^{-1}$ and the convention \eqref{eq:convention} is used for the action of the operators on the r.h.s.
\end{lemma}
\proof We use  the same arguments as in the proof of \cite[Prop.~3.5]{dun1} (note that the assumption $\nabla_n n=0$ made in \cite{dun1} is not needed these particular steps). We denote by $\p_T$ the time derivative in a local coordinate, then $n=-c(t)\p_T$. For $\psi\in C^\infty(M;S^+M)$, we have
\begin{align*}
(\p_t\circ U(t) \psi)(t)&=\lim_{\epsilon\to 0}\epsilon^{-1}(U(t+\epsilon)\psi|_{\Sigma_{t+\epsilon}}-U(t)\psi|_{\Sigma_t})\\
&=\lim_{\epsilon\to 0}\epsilon^{-1}(\rho(t+\epsilon)\tau_{t+\epsilon}\psi|_{\Sigma_{t+\epsilon}}-\rho(t)\tau_t\psi|_{\Sigma_t})\\
&=\tau_t \big(\lim_{\epsilon\to 0}\epsilon^{-1}(\tau_{t+\epsilon}^t\rho(t+\epsilon)\psi|_{\Sigma_{t+\epsilon}}-\rho(t)\psi|_{\Sigma_t})\big)\\
&=\tau_t \big( (\nabla_{\p_T}^{\SM}\rho(t) \psi)(t)\big)\\
&=\big(\rho(t)^{-1}U(t)\circ \nabla_{\p_T}^{\SM}\circ\rho(t)\psi\big)(t),
\end{align*}
where we used $\tau_t=\rho(t)^{-1}U(t)$ in the last step. Hence we have
\begin{align*}
\p_t&=\rho(t)^{-1}U(t)\circ \nabla_{\p_T}^{\SM}\circ\rho(t)U(t)^{-1}\\
&=-\rho(t)^{-1}U(t)\circ c(t)\nabla_{n}^{\SM}\circ\rho(t)U(t)^{-1},
\end{align*}
which implies
\beq
c(t)U(t) (- \nabla_n^{\SM} )U(t)^{-1}=\rho(t)\circ\p_t\circ\rho(t)^{-1}.\label{h1}
\eeq
In the next step we find a convenient expression for $r(t)$. On the one hand,
$$
\div n=\sum_{j=0}^d g(e_j,\nabla_{e_j}n)=g(n,\nabla_n n)+\sum_{j=1}^d g(e_j,\nabla_{e_j}n)=2r(t).
$$
On the other hand,
\begin{align*}
\div n|_{\Sigma_t}&=-{|g|}^{-\12}\p_t({|g|}^{\12} c(t)^{-1})=-c(t)^{-1}{|h(t)|}^{-\12}\p_t({|h(t)|^\12})\\
&=-2c(t)^{-1}|h(t)|^{-\frac{1}{4}}\p_t(|h(t)|^{\frac{1}{4}})=-2c(t)^{-1}\rho(t)^{-1}(\p_t\rho(t)).
\end{align*}
So we obtain
$
r(t)=-c(t)^{-1}\rho(t)^{-1}(\p_t\rho(t)).
$
In consequence, we have
\beq
c(t)U(t)(-r(t) )U(t)^{-1}=\rho(t)^{-1}\p_t(\rho(t)).\label{h3}
\eeq
Inserting \eqref{h1} and \eqref{h3} into the formula \eqref{eq:defD1} for $D$, we obtain \eqref{eq:equivalence}. \qed

\section{Time-dependent pseudo-differential calculus}\label{sec:pdo}

\subsection{Pseudo-differential operators}  We now introduce various classes of pseudodifferential operators that will be needed in the proofs. 

For $\xi\in\rr^n$ we use the notation $\bra \xi\ket = (1+\module{\xi}^2)^{\12}$, in particular $\bra t \ket=(1+t^2)^{\12}$ if $t\in\rr$. Recall that for $\ell\in \rr$, $S^\ell(\rr)$ is the space of all $a\in \cf(\rr)$ such that 
 \[
 \forall \gamma\in \nn, \ \exists \,C_{\gamma}\geq 0 \, \mbox{ s.t. } \big| \langle t\rangle^{-\ell+|\gamma|}\p_{t}^{\gamma}a(t)\big|\leq C_{\gamma} \mbox{ on  } \rr.
\]
If $U\subset \rr^n$ is an open set and $m\in\rr$,  the symbol space $S^m(T^*U)$ is the space of all $a\in  \cinf(U;T^{*}U)$ such that:
 \[
 \forall \alpha, \beta\in \nn^n, \ \exists \,C_{\alpha,\beta}\geq 0 \, \mbox{ s.t. } \big| \langle \spexi\rangle^{-m+|\beta|}\p_{\rx}^{\alpha}\p_{\spexi}^{\beta}a(\rx, \spexi)\big|\leq C_{\alpha,\beta} \mbox{ on  } U\times \rr^{d}.
\]
The best constants $C_{\alpha,\beta}$ define a family of semi-norms that allows to endow the space $S^m(T^*U)$ with a Fr\'echet space topology. 

We introduce a space of $t$-dependent symbols that behave in $t$ as a symbol of order $\ell\in\rr$. Namely, we define $S^{\ell}(\rr, S^{m}(T^{*}U))$ to be the space of all $a\in  \cinf(\rr\times U;T^{*}U)$ such that:
 \[
\forall \alpha, \beta\in \nn^{n}, \ \gamma\in \nn, \ \langle t\rangle^{-\ell+\gamma} \langle \spexi\rangle^{-m+|\beta|}\p_{t}^{\gamma}\p_{\rx}^{\alpha}\p_{\spexi}^{\beta}a(t,\rx, \spexi)\hbox{  is bounded on }\rr\times U\times \rr^{d}.
\]
This is equivalent to saying that $a \in\cf(\rr,S^m(T^*U))$ and for all $\gamma \in \nn$, the seminorms of $\langle t\rangle^{-\ell+\gamma} \p_t^\gamma a(t,\cdot,\cdot)$ in $S^m(T^*U)$ are bounded on $\rr$. More generally, we make the following definition, following \cite[\S{15.1}]{G}.

\bed\label{deftd} Suppose $\cF$ is a Fr\'echet space, and let $\left\| \cdot\right\|_j$, $j\in \nn$, be the family of seminorms that defines its topology. If $I\subset \rr$ is an open interval and $\ell\in\rr$, we define  $S^\ell(I,\cF)$ to be the space of smooth functions $I\ni t \mapsto a(t)\in   \cF$ such that
$$
\forall {j,\gamma\in \nn}, \  \sup_{t\in I}\,\langle t\rangle^{-\ell+\gamma}\left\| \p_t^\gamma  a(t) \right\|_j < \infty.
$$
\eed

Next, we consider the setting of a $k$-dimensional vector bundle $E$ over a compact manifold $\Sigma$.  If $\pi:T^*\Sigma\to \Sigma$ is the bundle projection, let $\pi^* \End(E)\to T^*\Sigma$ be the pullback bundle of $\End(E)\to \Sigma$  by $\pi$. A trivialisation $\varphi_U : \End(E)|_U\cong U\times \rr^{k^2}$  over an open set $U\subset \Sigma$ induces a trivialisation
\beq\label{eq:trivialf}
\varphi_{\pi,U} :\pi^*\End(E)|_ {T^*U}\cong T^*U\times\rr^{k^2}
\eeq
 of the vector bundle $\pi^*\End(E)$ over $T^*_U \Sigma$.  The symbol space $S^m(T^*\Sigma; \pi^*\End(E))$ is by definition the space of all $\cf(T^*\Sigma; \pi^*\End(E))$ which in local coordinates and in a trivialisation of the form \eqref{eq:trivialf} are elements of $S^m(T^*U)\otimes\rr^{k^2}$. The seminorms of $S^m(T^*U)\otimes\rr^{k^2}$ for different charts define a Fr\'echet space topology on $S^m(T^*\Sigma; \pi^*\End(E))$. The \emph{principal symbol} of $a\in S^m(T^*\Sigma; \pi^*\End(E))$   is the equivalence class
 $$
 [ a ] \in S^m(T^*\Sigma; \pi^*\End(E)) / S^{m-1}(T^*\Sigma; \pi^*\End(E)).
 $$

 Let $L^2(\Sigma;E)$ be the $L^2$ space defined using some positive, smooth section of the density bundle tensored with $E\otimes E^*$.

Let $\Psi^m(E)$ be the standard class of pseudo-differential operators on $E\to \Sigma$ and let us fix a quantization map
$$
S^m(T^*\Sigma; \pi^*\End(E)) \ni a \mapsto \Op(a)\in\Psi^m(E)
$$
with the extra property that ${\Red\Op({a}^*)}=\Op(a)^*$ in the sense of formal adjoint w.r.t.~the $L^2(\Sigma;E)$ inner product. The topology of $S^m(T^*\Sigma; \pi^*\End(E))$ can be used to topologize $\Psi^m(E)$.

We can then consider classes of functions in $t$ with values in symbols or pseudo-differential operators using the notation introduced in Definition \ref{deftd}. Note that by applying the quantization map $\Op$ pointwise to elements of  $S^\ell(I,S^m(T^*\Sigma; \pi^*\End(E)))$ we obtain the $t$-dependent pseudo-differential operators $S^\ell(I,\Psi^m(E))$.  For the sake of having more shorthand notation we set
$$
\pdo{m}{\ell}{I}\defeq S^\ell(I,\Psi^m(E)),
$$
and
$$
\sym{m}{\ell}{I}\defeq S^\ell(I,S^m(T^*\Sigma; \pi^*\End(E))
$$
for the corresponding time-depending symbols, where the subscript `$\rm td$' stands for  `time-decay'.  This class generalizes the time-depending   pseudo-differential operators introduced in the scalar case  in \cite{inout}, though an important simplification for us is that we consider only compact base manifolds (see \cite{Gerard2021a,Stoskopf} for the non-compact setting).

As usual in pseudo-differential calculus, we denote
$$
\pdo{\infty}{\ell}{I}\defeq \bigcup_{m\in \rr}\pdo{m}{\ell}{I}, \quad \pdo{-\infty}{\ell}{I}\defeq \bigcap_{m\in \rr}\pdo{m}{\ell}{I}
$$
and similarly for the $t$-independent classes.

\bed One says that $A\in \Psi^m(E)$ is \emph{principally scalar} if its principal symbol has a representative which is a $\cf(\Sigma;T^*\Sigma)$ multiple of the identity in $\End(E)$. We say that $A(\cdot)\in \pdo{m}{\ell}{I}$ is principally scalar if $A(t)$ is principally scalar for all $t\in \rr$.
\eed

We write $A(t)$ instead of $A(\cdot)$ when it is unlikely to cause any confusion. If $A(t)\in    \pdo{m_1}{\ell_1}{I}$ and $B(t)\in    \pdo{m_2}{\ell_2}{I}$ for some $m_i,\ell_i\in \rr$, $i=1,2$, then $$A(t)B(t)\in  \pdo{m_1+m_2}{\ell_1+\ell_2}{I}.
$$ Furthermore, if $A(t)$ or $B(t)$ is principally  scalar then
$$
[ A(t), B(t)] \in  \pdo{m_1+m_2-1}{\ell_1+\ell_2}{I}.
$$

\subsection{Beals type commutator criterion} We show that we can characterize operators in $\pdo{m}{\ell}{I}$ by a Beals type commutator criterion.

For $s,r\in\rr$, the norm in $\B{s}{r}$ is denoted $\norm{\cdot}_{B(H^s,H^{r})}$.

\bep\label{beals}
Suppose $A(t)\in \B{s}{s-m}$ for all $t\in I$, $s\in\rr$ and some $m\in\rr$. Then $A(\cdot)\in\pdo{m}{\ell}{I}$   if and only if $A(\cdot)$ is infinitely differentiable and
\beq\label{comm1}
\sup_{t\in I} \,\bra t \ket^{-\ell+\gamma} \| \ad_{L_1}\cdots \ad_{L_k} \p_t^\gamma A(t)\|_{B(H^s,H^{s-m+d(k)})}<+\infty
\eeq
for all $k\in\nn$, all  {\Red principally} scalar $L_1,\dots,L_k\in \Psi^1(E)$ and all $\gamma\in\nn$, where $d(k)=\sum_{j=0}^k (1-\deg L_j)$, $\deg L_j$ is the order of $L_j$, and we denoted $\ad_{L_j}B\defeq [L_j,B]$.
\eep
\proof For fixed $t\in I$, the Beals criterion on compact manifolds (see  \cite[\S5.3]{Ruzhansky2010}) generalizes to the vector bundle case, namely, it says in our situation that $A(t)\in\Psi^m(E)$ if and only if
\beq\label{comm2}
 \| \ad_{L_1}\cdots \ad_{L_k} \p_t^\gamma A(t)\|_{B(H^s,H^{s-m+d(k)})} <+\infty
\eeq
for all $k\in \nn$ and all {\Red principally} scalar $L_1,\dots,L_k\in\Psi^1(E)$. {\Red Indeed, in an arbitrarily chosen local trivialization we can  consider  \eqref{comm2} for the subclass of $L_i\in \Psi^1(E)$ which are scalar in that particular trivialization: this  then implies   $A(t)\in\Psi^m(E)$ by the scalar Beals criterion, applied in this trivialization.  The reverse implication is straightforward.}
   
Furthermore, it follows from the proof of \cite[Thm.~5.3.1]{Ruzhansky2010} that the seminorms $\norm{a(t)}_i$ of the symbol of $a(t)\in S^m(T^*\Sigma;\pi^*\End E)$  of $A(t)$ are bounded by \eqref{comm2} uniformly in $t\in I$, and the analogous property holds true for the seminorms $\p_t^\gamma a(t)$ of $\p_t^\gamma A(t)$. Thus,   \eqref{comm1} implies $\sup_{t\in I} \bra t \ket^{-\ell+\gamma}  \norm{\p_t^\gamma a(t)}_i<+\infty$, which means that $A(\cdot)\in \pdo{m}{\ell}{I}$. The reverse implication is straightforward. \qed

\begin{remark} As in the standard Beals commutator criterion, it suffices to check \eqref{comm1} for all \emph{differential} operators $L_1,\dots,L_k$ of order $\leq 1$.
\end{remark}

\subsection{Ellipticity} The principal symbol of $A=\Op(a)\in\Psi^m(E)$ is by definition  the  principal symbol $\sigma_m(A)\defeq [a]$ of $a$, and $A$ is \emph{elliptic} if there exists a symbol  $b\in S^{-m}(T^*\Sigma; \pi^*\End(E))$ such that
$$
ab-\one, ba-\one \in S^{-1}(T^*\Sigma; \pi^*\End(E)).
$$
We say that $A(\cdot)\in \pdo{m}{\ell}{I}$ is elliptic if $A(t)$ is elliptic for all $t\in \rr$.

If $A\in\Psi^1(E)$ is positive, elliptic and principally scalar, one can define the Sobolev space of order $s\in \rr$ by
$$
H^s(\Sigma;E) \defeq A^{-s} L^2(\Sigma;E)
$$
equipped with the norm $\| u \|_{H^s}\defeq \| A^s u  \|_{L^2}$, where $\| u \|_{L^2}$ is the $L^2(\Sigma;E)$ norm of $u$.  Note that choosing a different $A$ with the properties mentioned above defines in general a different, but equivalent norm.

The operators of interest to us will typically not be principally scalar, but will often satisfy the following condition instead {\Red(alongside self-adjointness properties discussed in the sequel)}.

\bed We say that $H\in \Psi^m(E)$ is \emph{of Dirac type} if it is elliptic and $H^2$ is principally scalar. We say that $H(\cdot)\in \pdo{m}{\ell}{I}$ is of Dirac type if $H(t)$ is of Dirac type  for all $t\in \rr$.
\eed

\subsection{Families similar to self-adjoint operators} Let $\delta>0$  and let us fix {\Red a family of zero order pseudodifferential operators $T(t)$, boundedly invertible uniformly in $t\in\rr$ and satisfying:  }
\beq\label{eq:tdp}
\Red \cp(t), \cm(t)\in  \pdo{0}{0}{\rr}, \quad \p_t \cp(t)\in  \pdo{0}{-1-\delta}{\rr}.
\eeq
Suppose $H(t)\in \pdo{1}{0}{\rr}$ is a family of elliptic operators of the form
\beq\label{condh}
H(t)=\cp  (t) H_0(t) \cm (t), \mbox{ with } H_0(t)\in \pdo{1}{0}{\rr}  \mbox{ s.t.}\,\,\forall t\in\rr, \ (H_0(t))^*=H_0(t).
\eeq
Then, $H_0(t)$ is elliptic as well and as a consequence it is self-adjoint with domain $H^1(\Sigma;E)$ and has discrete spectrum (see e.g.~\cite[Thm.~8.3]{shubin}). 
Since multiplication by $\cp(t), \cm(t)$ preserves $H^1(\Sigma;E)$, $H(t)$ with domain $H^1(\Sigma;E)$ is closed and  it is similar to a self-adjoint operator and has therefore \emph{discrete, real spectrum}. Furthermore, by \eqref{condh} and self-adjointness of $H_0(t)$ we also conclude that for all $\lambda\in\cc\setminus\rr$ and $N\in\nn_{\geq 0}$,
\beq\label{resolventbound}
\bea
\norm{(H(t)-\lambda)^{-N}}&= \big\|\cm (t) \big( H_0(t)-\lambda)^{-N}  \cp (t)  \big\| \\
&\leq C \|  \big(  H_0(t) -\lambda)^{-N} \|\\
&\leq C \module{\Im \lambda}^{-N}
\eea
\eeq
uniformly in $t\in\rr$. By the Hille--Yosida theorem (see e.g.~\cite[\S1.3]{pazy}), for each $t\in\rr$, $i H(t)$  is the generator of a bounded semi-group.

\begin{remark} An equivalent point of view is to consider  the operators $H(t)$ as self-adjoint operators in the $t$-dependent Hilbert space  denoted in the sequel $L^2_t(\Sigma;E)$ and defined using the norm  $\| u \|_{L^2_t}\defeq  \| \cm(t) u\|_{L^2}$ for  $t\in\rr$. 
\end{remark}

\subsection{Functional calculus}\label{ss:funcalcul} In view of \eqref{condh} and \eqref{resolventbound} it is possible to define a functional calculus for $H(t)$ with good properties. In what follows we will use the calculus based on the Helffer--Sjöstrand formula as introduced by Davies \cite{davies} for an even more general class of operators.

Namely, for $f\in S^\rho(\rr)$ with $\rho<0$, one  defines
 \beq\label{eq:helffersjostrand}
 f(H(t))=\frac{1}{2\pi i}\int_{\mathbb{C}}\frac{\p\tilde{f}}{\p \bar{z}}(z)\big(z-H(t)\big)^{-1} d\bar{z}\wedge dz,
 \eeq
 where $\tilde{f}$ is an $N$-th order \emph{almost analytic extension of $f$} of the form
 $$
 \tilde f(x+i y) = \bigg( \sum_{k=0}^N f^{(k)}(x) \frac{(i y)^k}{k!}\bigg)\psi\Big(\frac{y}{\bra x \ket} \Big)
 $$
 for some $\psi\in \cf_{\rm c}(\rr,[0,1])$ with $\psi\equiv 1$ on $[-1,1]$ and $\psi\equiv 0$ outside of $[-2,2]$, and $N\in \nn_{\geq 0}$ is taken sufficiently large.  The crucial properties satisfied by $\tilde f$ are:
 $$
 \bea
 {} & \tilde f |_{\rr} = f, \quad \bigg|\frac{\p\tilde{f}}{\p \bar{z}}\bigg|\leq C \bra x\ket ^{\rho-1-N} \module{y}^N, \\
  & \supp\tilde f \subset \{ x+iy\in \cc \st |y| \leq 2 \bra x \ket, \, x \in \supp f \},
\eea
 $$
which together with the resolvent bound \eqref{resolventbound} ensure that \eqref{eq:helffersjostrand} is well-defined. If $\rho\geq 0$ then the proof of  \cite[Prop.~2.5]{sylvain} shows that
\beq\label{betterhelffer}
 f(H(t)) u \defeq \frac{1}{2\pi i}\lim_{R\to+\infty}\int_{\module{\Re z} < R}\frac{\p\tilde{f}}{\p \bar{z}}(z)\big(z-H(t)\big)^{-1} u \,d\bar{z}\wedge dz
\eeq
exists for all $u$ in the domain of $\bra H(t)\ket^\rho$. We also remark that for each $t\in\rr$ we can easily switch to the functional calculus for self-adjoint operators and back thanks to the  relation  $(H-\lambda)^{-1}=\cm   \big( H_0 -\lambda)^{-1}  \cp $.

 In \sec{sec:scattering}  we will need the following two essential results, Theorem \ref{thmseeley} and Proposition \ref{propseeley}  which generalize the scalar case given in \cite[Thm.~3.7]{inout} (this is a variant of \emph{Seeley's theorem} \cite{seeley}), resp.~\cite[Prop.~3.10]{inout}. Using the principally scalar assumption and the fact that  multiplying $\cp(t),\cm(t)$ preserves the spaces $\pdo{m}{\ell}{I}$,  the proofs are fully analogous to \cite{inout}.

 \begin{theorem}\label{thmseeley}  For $m>0$, if $A(t)\in\pdo{m}{0}{I}$ is principally scalar, elliptic and $\cp (t)A(t)\cm (t)\geq C\one$ for some $C>0$ and all $t\in I$, then $A^s(t)\in \pdo{ms}{0}{I}$ and $A^s(t)$ is principally scalar and elliptic for all $s\in\rr$.
 \end{theorem}

 \begin{proposition}\label{propseeley}  For $m>0$, if  $A_1(t),A_2(t)\in\pdo{m}{0}{I}$ are principally scalar, elliptic, and satisfy $\cp (t)A_i(t)\cm (t)\geq C_i \one$ for some $C_i>0$, $i=1,2$, and $A_1(t)-A_2(t)\in \pdo{m}{\ell}{I}$ for some $\ell<0$, then $A_1^s(t)-A_2^s(t)\in \pdo{ms}{\ell}{I}$ for all $s\in\rr$.
 \end{proposition}


\section{Evolutionary model and scattering}\label{sec:scattering}

\subsection{Model Dirac equation}\label{ss:model} As in \sec{sec:pdo}, we consider the setting of a Hermitian bundle $E$ over a compact manifold $\Sigma$.

We consider an equation of the form
$$
(\p_t  - i H(t) ) \phi(t) = 0,
$$
where $H(t)\in\pdo{1}{0}{\rr}$ satisfies the following hypothesis.

\begin{assumption}\label{hyptd} There exists $\delta>0$ and $H_\pm\in \Psi^1(E)$ of Dirac type such that:
\ben
\item\label{td1} $H(t)\in \pdo{1}{0}{\rr}$ is of Dirac type and $H(t)-H_\pm  \in \pdo{1}{-\delta}{\rr_\pm}$,
\item\label{td1b}   $H(t)$ is   of the form \eqref{condh} for some $\cp(t)$ satisfying \eqref{eq:tdp}.
\een
\end{assumption}

Using the functional calculus introduced in \sec{ss:funcalcul},  we define the family
\beq\label{deflam}
\rr\ni t \mapsto  \Lambda(t)\defeq (1+H^2(t))^{\12}
\eeq
which will serve as a reference family of elliptic  operators. By the assumption that $H(t)$ is of Dirac type, $\Lambda(t)$ is  principally scalar. More precisely, we show the following statement.

\begin{lemma}  \label{lemlam} The family $\Lambda(t)$ defined in \eqref{deflam} belongs to $\pdo{1}{0}{\rr}$, is principally scalar, and satisfies for all $s\in\rr$:
\ben
\item\label{lemlam2}  $\p_t\Lambda^s(t)\in \pdo{s}{-1-\delta}{\rr}$;
\item\label{lemlam3}  there exists $C_{1,s},C_{2,s}>0$ such that $C_{1,s} \| u \|_{H^s}\leq    \|  \Lambda^s(t) u \|_{L^2} \leq C_{2,s} \| u \|_{H^s}$ for all $u\in H^s(\Sigma;E)$ and $t\in \rr$.
\een
\end{lemma}

\proof The  fact  that $\Lambda(t)\in\pdo{1}{0}{\rr}$ (and that is is principally scalar) follows from the assumption that $H(t)$ is of Dirac type and  from Theorem \ref{thmseeley}. The latter also gives that $\Lambda^s(t)\in \pdo{s}{0}{\rr}$ and $\Lambda^s(t)$ is principally scalar, which proves \eqref{lemlam3}.

Let us prove \eqref{lemlam2}. By  definition of $\Lambda(t)$,
$$
\Lambda^2(t)- (1+H_\pm^2)= H^2(t) -H_\pm^2= (H(t) +H_\pm)(H(t) -H_\pm)- [H_\pm,H(t)-H_\pm].
$$
By  \eqref{td1} of Hypothesis \ref{hyptd}, this belongs to $\pdo{2}{-\delta}{\rr_\pm}$. Therefore, we can apply Proposition \ref{propseeley} to the families $\Lambda^2(t)$ and $1+H_\pm^2$, which gives
$$
\Lambda^s(t)- (1+H_\pm^2)^{\frac{s}{2}}\in \pdo{s}{-\delta}{\rr_\pm}.
$$
By differentiating  we obtain $\p_t\Lambda^s(t) \in  \pdo{s}{-1-\delta}{\rr_+}$, $\p_t\Lambda^s(t) \in  \pdo{s}{-1-\delta}{\rr_-}$, hence $\p_t\Lambda^s(t) \in  \pdo{s}{-1-\delta}{\rr}$. \qed

\subsection{Schr\"odinger propagator of $H(t)$} We will  now be interested in the evolution generated by $H(t)$.

\bed  Given a family of operators $\rr\ni t \mapsto A(t)\in B(H^1(\Sigma;E), L^2(\Sigma;E))$, we say that $\rr^2\ni (t,s)\mapsto U(t,s)\in B(L^2(\Sigma;E))$ is a \emph{Schr\"odinger propagator of $A(t)$} if for all $t,t_0,s\in \rr$ it satisfies:
\ben
\item $U(t,t)=\one$;
\item $U(t,t_0)U(t_0,s)=U(t,s)$;
\item $U(t,s)$ is strongly continuously differentiable in $B(H^1(\Sigma;E), L^2(\Sigma;E))$, it preserves $H^1(\Sigma;E)$, and
\beq
\p_t U(t,s)= i A(t) U(t,s),  \quad
\p_s U(t,s)=-U(t,s) i A(s).
\eeq
\een
\eed

\begin{proposition}\label{proprog} The Schr\"odinger propagator of $H(t)$, denoted by $U(t,s)$, exists and is unique. Furthermore, $\{ U(t,s)\}_{t,s\in\rr}$ is uniformly bounded in $B(L^2(\Sigma;E))$.
\end{proposition}
\proof  The family  $t\mapsto H(t)$ is bounded in $B(H^1(\Sigma;E), L^2(\Sigma;E))$ and differentiable.  Furthermore, for each $t\in\rr$, $H(t)$ is a closed operator in the sense of the Hilbert space $L^2(\Sigma;E)$, with $t$-independent domain $H^1(\Sigma;E)$. It is also the generator of a strongly continuous bounded semi-group. In consequence, we can apply  \cite[\S 5,~Thm.~4.8]{pazy}  (with $M=1$ and $\omega=0$ therein) to conclude the assertion.  \qeds

Note that the Schr\"odinger propagator of $H_\pm$ is simply $e^{i(t-s)H_\pm}$, and the ``free'' evolution $e^{i(t-s)H_\pm}$ is the obvious comparison dynamics for the scattering theory at $t\to\pm\infty$. However, for our purpose it will often be better to consider the family $\{e^{itH(t)}\}_{t\in\rr}$, which solves
\beq\label{eq:diffht}
\p_t e^{itH(t)}= i (H(t)+t \p_t  H(t)) e^{itH(t)}.
\eeq

Because  $e^{i(t-s)H_\pm}$ commutes with $(1+H^2_\pm)^\12$, it is straightforward to show that the family $\{e^{i(t-s)H_\pm}\}_{t,s\in\rr}$ is bounded in $B(H^m(\Sigma;E))$ for all $m\in\rr$.  Thanks to our decay assumptions and Lemma \ref{lemlam}  we also have the following result.

\begin{lemma}\label{lembd} For all $m\in \rr$,  $\{ U(t,s)\}_{t,s\in\rr}$ and $\{e^{itH(t)}\}_{t\in\rr}$ are uniformly bounded  in $B(H^m(\Sigma;E))$.
 \end{lemma}
\proof  In view of Lemma \ref{lemlam} and Proposition \ref{proprog}, we can repeat the second paragraph of \cite[Proof of Prop.~5.6]{inout} verbatim  (with $\Lambda(t)$ instead of $\epsilon(t)$,  and with $H(t)$ and $U(t,s)$ instead of $H^{\rm ad}(t)$ and $U^{\rm ad}(t,s)$). The same proof gives also the uniform boundedness of $\{e^{itH(t)}\}_{t\in\rr}$. \qed

\subsection{\texorpdfstring{Functions of $H(t)$ as pseudodifferential operators}{Functions of H(t) as pseudodifferential operators}} Using the Beals type criterion formulated in Proposition \ref{beals}, we show that suitable functions of $H(t)$ are pseudo-differential operators  in the sense of our time-dependent classes.

\begin{lemma}\label{ptH0t}
For all $\chi\in S^0(\rr)$, $\chi(H(t))\in \pdo{0}{0}{\rr}$. Furthermore, $\p_t \chi(H(t)) \in \pdo{0}{-1-\delta}{\rr}$.
\end{lemma}
\proof \step{1} In the first step we analyse the resolvent $(H(t)-z)^{-1}$ for $z\in \cc\setminus \rr$.

Let $\cR^\infty(z)$ be the algebra finitely generated by $(H(t)-z)^{-1}$ and by ($z$-independent) elements of $\pdo{m}{0}{\rr}$ for all $m\in\rr$ under operator product and sum. Let us also denote
$$
\cR^m(z) \defeq \cR^\infty(z) \cap \big( \textstyle\bigcup_{s\in\rr} \B{s}{s-m}\big).
$$
This defines a filtration of $\cR^\infty(z)$ satisfying $\cR^{m_1}(z)\cR^{m_2}(z)\subset \cR^{m_1+m_2}(z)$. For all {\Red principally} scalar $L_i\in \Psi^1(M;E)$, we have
$$
\bea
\ad_{L_i} \pdo{m}{0}{\rr} \subset \pdo{m-(1-\deg L_i)}{0}{\rr} \subset \cR^{m-(1-\deg L_i)}(z)
\eea
$$
and
$$
\bea
\ad_{L_i} (H-z)^{-1} = (H-z)^{-1} [ H, L_i ](H-z)^{-1} & \in (H-z)^{-1} \cR^{-(1-\deg L_i)}(z)\\
& \phantom{\in\,}\subset \cR^{-1 -(1-\deg L_i)}(z).
\eea
$$
Since $\ad_{L_i}$ acts as a derivation,  in greater generality we have
$$
 \ad_{L_i} \cR^m(z)\subset    \cR^{m-(1-\deg L_i)}(z),
$$
and
$$
\bea
\ad_{L_i} (H-z)^{-1} \cR^m(z) &\subset   \big(\ad_{L_i} (H-z)^{-1}\big) \cR^m(z) +(H-z)^{-1} \ad_{L_i}  \cR^m(z)\\
& \subset (H-z)^{-1}  \cR^{m-(1-\deg L_i)}(z) + (H-z)^{-1}  \cR^{m-(1-\deg L_i)}(z)\\
&  \subset  (H-z)^{-1}  \cR^{m-(1-\deg L_i)}(z).
\eea
$$
By iterating this $k\in \nn$ times,  for all {\Red principally} scalar $L_1,\dots,L_k\in \Psi^1(M;E)$ we obtain:
\beq\label{itera}
\ad_{L_1}\cdots \ad_{L_k} (H-z)^{-1}  \cR^m(z) \subset (H-z)^{-1} \cR^{m-d(k)} (z)
\eeq
where $d(k)=\sum_{j=0}^k (1-\deg L_j)$.

On the other hand, as a consequence of  the resolvent identity
$$
(H(t+h)-z)^{-1} - (H(t)-z)^{-1}= (H(t)-z)^{-1}( H(t) - H(t+h) )  (H(t+h)-z)^{-1}
$$
we have
\beq\label{eq:llk}
\p_t ( H(t) - z )^{-1} =   - ( H(t) - z )^{-1}( \p_t H(t) )  ( H(t) - z )^{-1}.
\eeq
Using \eqref{eq:llk} and $\p_t H(t)\in \pdo{1}{-1}{\rr}$ repeatedly, for all $\gamma\in\nn$ we get
$$
\bra t \ket^{\gamma} \p_t^\gamma  ( H(t) - z )^{-1} \in   ( H(t) - z)^{-1}  \cR^0(z).
$$
In consequence, we can apply \eqref{itera} with $m=0$, which yields
$$
 \bra t \ket^{\gamma} \ad_{L_1}\cdots \ad_{L_k}  \p_t^\gamma (H-z)^{-1}   \in (H-z)^{-1} \cR^{-d(k)}(z)\subset \cR^{-1-d(k)}(z).
$$
Considering the definition of $\cR^{-1-d(k)}(z)$ and  boundedness properties of pseudo-differential operators, this means that
\beq\label{eqbeals1}
\sup_{t\in \rr} \,\bra t \ket^\gamma \norm{\ad_{L_1}\cdots \ad_{L_k}  \p_t^\gamma (H(t)-z)^{-1}}_{B(H^s,H^{s+1+d(k)})} < +\infty.
\eeq
for all $s\in\rr$.
Alternatively, using the uniform estimate
$$
\norm{(H(t)-z)^{-1}}_{B(H^{s},H^{s})}  \lesssim \module{\Im z}^{-1}
$$
one time, we obtain in a similar way
\beq\label{eqbeals2}
\sup_{t\in \rr} \,\bra t \ket^\gamma \norm{\ad_{L_1}\cdots \ad_{L_k}  \p_t^\gamma (H(t)-z)^{-1}}_{B(H^s,H^{s+d(k)})} \lesssim \module{\Im z}^{-1}.
\eeq
By the Beals-type commutator criterion stated in Proposition \ref{beals}, we conclude from \eqref{eqbeals1}  that $(H-z)^{-1}\in\pdo{-1}{0}{\rr}$ and from   \eqref{eqbeals2} that the seminorms of $(H-z)^{-1}$ in $\pdo{0}{0}{\rr}$ are $O(\module{\Im z})^{-1}$.

\step{2} In the second step, using the Helffer--Sjöstrand formula we deduce $\chi(H(t))\in \pdo{0}{0}{\rr}$ from what we already know on the resolvent.  Namely we apply \eqref{betterhelffer} in the form
\beq\label{HSfo}
\chi\big( H(t)\big)=\slim_{R\to +\infty} \frac{1}{2\pi i}\int_{\mathbb{C}}\p_{\bar{z}}\widetilde{(\chi \psi_R)}(z)(z-H(t))^{-1}  d\bar{z} \wedge dz,
\eeq
where $\psi_R(t)\defeq\psi(t/R)$, $\psi\in C^\infty_\c(\rr,[0,1])$ with $\psi=1$ near 0, the limit is taken in the strong operator topology and $\widetilde{(\chi\psi_R)}$ is an almost analytic extension of $\chi\psi_R$ of order $N\geq 2$, in particular
\beq\label{Ngeq2}
\widetilde{(\chi\psi_R)}|_{\rr}=\chi\psi_R,\qquad |\p_{\bar{z}}\widetilde{(\chi\psi_R)}(z)|\leq c\bra \Re z\ket ^{-1-N}\left|\Im z\right|^N.
\eeq
Above, $c$ depends only on the semi-norms
$$
\sup_{t\in \rr}\,\langle t\rangle^{k}|(\chi\psi_R)^{(k)}(t)|,\qquad k\in\nn,
$$
see \cite[\S2]{sylvain} for more details. We analyse the dependence on $R$:
\begin{align*}
\sup_{t\in \rr}\,\langle t\rangle^{k}|(\chi \psi_R)^{(k)}(t)| & \lesssim  \sup_{t\in\rr}\sum_{m=0}^k \bra t \ket^{k} \chi^{(m)}(t) |(\psi_R)^{(k-m)}(t)| \\
& \lesssim \sum_{m=0}^k \sup_{t\in\rr} \,\bra t \ket^m \chi^{(m)}(t)\cdot \sup_{t\in\rr}\,\bra t \ket^{k-m} R^{m-k}|\psi^{(k-m)}(t/R)| \\
&= O(1),
\end{align*}
where we used the fact that $|\psi^{(k)}|\in\cinf_\c(\rr)$ and $\chi\in S^0(\rr)$ in the last step.  Therefore, we can use \eqref{HSfo} and the $O(\module{\Im z})^{-1}$  bound on seminorms of $(H-z)^{-1}$ in $\pdo{0}{0}{\rr}$  proved in \emph{Step 1}  to conclude  that $\chi\big(H(t)\big)\in\pdo{0}{0}{\rr}$.

\step{3} To show that $\p_t \chi(H(t)) \in \pdo{0}{-1-\delta}{\rr}$ we  repeat \emph{Step 1} and \emph{Step 2}. The extra decay is obtained by taking into account  that $\p_t H(t) \in \pdo{1}{-1-\delta}{\rr}$ (which follows from    Hypothesis \ref{hyptd}) when using \eqref{eqbeals1} .
\qeds

More generally, for $s\in\rr$, any $f\in S^s(\rr)$ is of the  form $f(x)=\bra x\ket^s \chi(x)$ for some $\chi\in S^0(\rr)$. Since  by Proposition \ref{propseeley}, $\bra H(t)\ket^s \in \pdo{s}{0}{\rr}$, we obtain the following immediate corollary.

\begin{corollary}\label{corcor} For all $s\in\rr$, if $f\in S^s(\rr)$ then $f\big( H(t) \big)\in \pdo{s}{0}{\rr}$.
\end{corollary}

This implies in particular that if $f\in S^s(\rr)$ with $s< 0$ then $f\big( H(t) \big)$ is compact as an operator in $\B{m}{m}$ for each $m\in\rr$.

\subsection{Large-time evolution of spectral projections}  In this part we  prove  the following key result on the large-time Heisenberg evolution of spectral projections.

\begin{proposition}\label{lemkey} Let $\chi\in S^0(\rr,[0,1])$ be such that $\chi=\one_{\clopen{0,+\infty}}$ outside of a bounded neighborhood of $0$. Then there exists   $R(t)\in \pdo{-1}{0}{\rr}$ such that for all $t\in\rr$,
\beq\label{eqitislc}
 U(t,0) \big(  \chi (H(0))   + R(0)  \big) U(0,t) =\chi(H(t)) + R(t),
\eeq
and furthermore,
\beq\label{slims}
\slim_{t\to \pm \infty} U(0,t)R(t)U(t,0)\in \Psi^{-\infty}(E).
\eeq
\end{proposition}

Note that in contrast to the scalar case, $U(t,0)R(0)U(0,t)$ is not necessarily pseudo-differential if $R(0)$ is pseudo-differential, so it cannot be simply absorbed into the r.h.s.~of \eqref{eqitislc} as  in the standard formulation of Egorov's theorem (see e.g.~\cite[\S11]{Zworski2012}).

Before giving the proof we need a couple of auxiliary results. For the sake of brevity we denote
$$
P_+(t)=\chi(H(t)), \quad P_-(t)=\one - \chi(H(t)).
$$

\begin{lemma}\label{lemego1} Let $m,\ell\in \rr$. Suppose $A(t)\in\pdo{m}{\ell}{\rr}$ satisfies
\beq\label{pap}
P_+(t) A(t) P_+(t)\in  \pdo{-\infty}{\ell}{\rr}, \quad P_-(t) A(t) P_-(t)\in  \pdo{-\infty}{\ell}{\rr}.
\eeq
Then there exists $Z(t)\in \pdo{m-1}{\ell}{\rr}$ such that
\beq\label{hra}
[H,Z](t)= A(t) \mod  \pdo{-\infty}{\ell}{\rr}.
\eeq
\end{lemma}
\proof \step{1} Let $k\in \rr$ and suppose $A_k(t)\in\pdo{k}{\ell}{\rr}$ satisfies \eqref{pap}. We will show that there exists $Z_{k-1}\in\pdo{k-1}{\ell}{\rr}$ such that
\beq\label{hraa}
[H,Z_{k-1}] = A_k - A_{k-1}
\eeq
for some $A_{k-1}(t)\in \pdo{k-1}{\ell}{\rr}$ satisfying \eqref{pap}.

By the hypothesis that $A_k(t)\in\pdo{k}{\ell}{\rr}$ satisfies \eqref{pap} we can write
$$
A_k= P_+ A_k P_-  +  P_- A_k P_ + \mod  \pdo{-\infty}{\ell}{\rr}.
$$
Let us set
$$
Z_{k-1}=\12 P_+ \F  A_k  P_- - \12  P_- \F  A_k  P_+ \in \pdo{k-1}{\ell}{\rr},
$$
where $\F(t) \in \pdo{-1}{0}{\rr}$ is defined as $\F(t)=f(H^2(t))$ for some $f\in S^{-\12}(\rr)$ such that $f(x)=\module{x}^{-\12}$  outside of a neighborhood of $0$ in $\rr$. Note that $\F(t)$ is principally scalar because $H(t)$ is of Dirac type. Furthermore, using that $x \one_{\clopen{0,+\infty}}(x) \module{x^2}^{-\12}=  \one_{\clopen{0,+\infty}}(x)$ away from zero, we obtain
\beq\label{eqtlllm}
 \bea
 &H P_+ \F=P_+ \mod \pdo{-\infty}{0}{\rr}, \\
  & H P_- \F= -P_-  \mod \pdo{-\infty}{0}{\rr},
 \eea
 \eeq
where we applied Corollary \ref{corcor} to conclude that the error term is in  $\pdo{-\infty}{0}{\rr}$. It also commutes with $H(t)$ and $P_\pm(t)$. Therefore,
 \beq\label{horoki1}
 \bea
  H Z_{k-1} & = \12 H P_+  \F A_{k} P_-  - \12 H P_-  \F A_{k} P_+ \\
  &=  \12 P_+  A_{k} P_-  + \12 P_- A_{k} P_+ \mod  \pdo{-\infty}{0}{\rr} \\
    &= \12 A_k \mod  \pdo{-\infty}{0}{\rr}.
 \eea
 \eeq
 Similarly,
  \beq\label{horoki2}
  \bea
    Z_{k-1} H & = \12 P_+ \F A_{k}H  P_-  - \12 P_-  \F A_{k} H  P_+ \\
   & = \12 P_+   A_{k} \F H P_-  - \12 P_-   A_{k}\F H P_+
   \fantom +  \12 P_+ [\F, A_{k}]H  P_-  - \12 P_-  [\F ,A_{k}] H  P_+\\
   &= - \12 P_+  A_{k} P_-  - \12 P_- A_{k} P_+ \fantom +  \12 P_+ [\F, A_{k}]H  P_-  - \12 P_-  [\F ,A_{k}] H  P_+ \mod  \pdo{-\infty}{0}{\rr} \\
     &= - \12 A_k   +   \12 P_+ [\F, A_{k}]H  P_-  - \12 P_-  [\F ,A_{k}] H  P_+ \\[-2mm] & \phantom{= - \12 A_k   +   \12 P_+ [\F, A_{k}]H  P_-  - \12 P_- P} \mod \pdo{-\infty}{0}{\rr}.
  \eea
  \eeq
   By substracting the two identities \eqref{horoki1} and \eqref{horoki2} we find that  $[H,Z_{k-1}]=A_k-A_{k-1}$ with $A_{k-1}$ satisfying \eqref{pap} as requested, and belonging to $\pdo{k-1}{\ell}{\rr}$ in view of $[\F, A_{k}]\in \pdo{k-2}{\ell}{\rr}$.

 \step{2} Let now $A_m=A$, and for $j\in\nn$, let $Z_{m-j}\in \pdo{m-j}{l}{\rr}$ and $A_{m-j}\in \pdo{m-j}{l}{\rr}$ be constructed recursively from $A_{m-j+1}$ using \emph{Step 1}. Let $Z\simeq \sum_{j=1}^{\infty} Z_{m-j}\in \pdo{m-1}{l}{\rr}$ be defined by asymptotic summation. Then $Z$ satisfies \eqref{hra} in view of \eqref{hraa}. \qed

\begin{lemma} \label{lemego2}
Let $m,\ell,\ell_0\in \rr$. Suppose $B(t)\in \pdo{m}{\ell_0}{\rr}$ solves
\beq\label{bhb0}
\p_t B(t) = i [H,B](t)+ E_m(t)
\eeq
for some $E_m(t)\in \pdo{m}{\ell}{\rr}$ satisfying \eqref{pap}.  Then there exists $R(t)\in \pdo{m-1}{\ell}{\rr}$ such that
\beq\label{brh}
\p_t (B(t)+R(t)) = i [H,B+R](t) \mod \pdo{-\infty}{\ell}{\rr}.
\eeq
\end{lemma}
\proof  \step{1} Let $k\in\nn_{0}$. Suppose for the moment that instead of \eqref{bhb0}, $B\in \pdo{m}{\ell_0}{\rr}$ satisfies
\beq\label{bhb}
\p_t B = i [H,B]+ P_+  E_{m-k} P_- +  P_-  E_{m-k} P_+ \mod \pdo{m-k-1}{\ell}{\rr}
\eeq
for some $E_{m-k}(t)\in \pdo{m-k}{\ell}{\rr}$. We will show that there exist $Z_{m-k-1}(t),S_{m-k-1}(t)\in \pdo{m-k-1}{\ell}{\rr} $ such that
\beq\label{bzh2}
\bea
\p_t (B+Z_{m-k-1} + S_{m-k-1})&=   i [H,B+Z_{m-k-1} + S_{m-k-1}] \fantom + P_+  E_{m-k-1} P_- +  P_-  E_{m-k-1} P_+ \mod \pdo{m-k-2}{\ell}{\rr}
\eea
\eeq
for some $E_{m-k-1}(t)\in \pdo{m-k-1}{\ell}{\rr}$.

First, using Lemma \ref{lemego1} we can find $Z_{m-k-1}(t)\in \pdo{m-k-1}{\ell}{\rr}$ such that
$$
i [H, Z_{m-k-1}] =P_+  E_{m-k} P_- +  P_-  E_{m-k} P_+ \mod \pdo{-\infty}{\ell}{\rr}.
$$
Inserting this into \eqref{bhb} we get
\beq\label{bzh}
\p_t (B+ Z_{m-k-1})= i [H, B+ Z_{m-k-1}] + E_{m-k-1}
\eeq
for some $E_{m-k-1}(t)\in \pdo{m-k-1}{\ell}{\rr}$. Next, we claim that we can find $S_{m-k-1}(t)\in \pdo{m-k-1}{\ell}{\rr}$ such that
\beq\label{shm}
\p_t S_{m-k-1} - i [H , S_{m-k-1}] = - P_+ E_{m-k-1} P_+ -  P_- E_{m-k-1} P_- \mod \pdo{m-k-2}{\ell}{\rr}.
\eeq
Indeed, it suffices to integrate the resulting ODE for the principal symbol of $S_{m-k-1}$, see e.g.~\cite[(4.17)--(4.18)]{Cordes1982}. Then, by adding the  identities \eqref{bhb} and \eqref{bzh} we obtain \eqref{bzh2} as wanted.

\step{2} Now, if $B$ satisfies \eqref{bhb0}, then it satisfies \eqref{bhb} with $k=0$. By iterating \emph{Step 1} indefinitely we obtain two sequences of operators $Z_{m-1-j}, S_{m-1-j}\in\pdo{m-1-j}{\ell}{\rr}$ for $j\in \nn_0$ such that the asymptotic sum $R\simeq \sum_{j=0}^\infty (Z_{m-1-j}+S_{m-1-j})\in \pdo{m-1}{\ell}{\rr}$ satisfies $\eqref{brh}$.\qeds

\refproof{Proposition \ref{lemkey}} Recall that we have denoted  $P_+(t)=\chi(H(t))$ and $P_-(t)=\one - \chi(H(t))$, and that $\p_t P_+(t)\in \pdo{0}{-1-\delta}{\rr}$ by Lemma \ref{ptH0t}. By differentiating
  the identity $(P_+(t))^2=P_+(t)$ we obtain
$$
P_\pm(t) \big(\p_t P_+(t) \big)  P_\pm(t) = 0.
$$
Furthermore, using that $H(t)$ commutes with $P_+(t)$ we can write
$$
\p_t P_+(t)= i [H(t),P_+(t)] + E_0(t)
$$
where $E_0(t)=\p_t P_+(t)\in\pdo{0}{-1-\delta}{\rr}$. Therefore, the assumption of Lemma \ref{lemego2} applied to $B(t)=P_+(t)$ are satisfied, with $m=\ell_0=0$ and $\ell=-1-\delta$. The lemma gives us the existence of $R_{-1}(t)\in\pdo{-1}{-1-\delta}{\rr}$ and $R_{-\infty}(t)\in\pdo{-\infty}{-1-\delta}{\rr}$ such that
$$
\p_t \big(P_+(t) + R_{-1}(t)\big)= i [ H (t),P_+(t)+ R_{-1}(t)] + R_{-\infty}(t).
$$
By Duhamel's principle we deduce that
\beq\label{duha}
\bea
U(t,0) \big( P_+(0) + R_{-1}(0)\big)  U(0,t)&=P_+(t)+R_{-1}(t)\fantom -  U(t,0)\bigg(\int_0^t U(0,s) R_{-\infty}(s)U(s,0) ds\bigg) U(0,t) \\
&\eqdef P_+(t)+R(t).
\eea
\eeq
With this notation, $R(t)-R_1(t)$ is bounded in $\B{s_1}{s_2}$ for all $s_1,s_2\in\rr$, uniformly in $t$ with all derivatives, therefore $R(t)-R_{-1}(t)\in \pdo{-\infty}{0}{\rr}$ and consequently $R(t)\in \pdo{-1}{0}{\rr}$. Furthermore, $R(0)=R_{-1}(0)$, which in combination with \eqref{duha} shows that
$$
U(t,0) \big( P_+(0) + R(0)\big)  U(0,t)=P_+(t)+R(t).
$$
as wanted. Finally, using the definition of $R(t)$ and the fact that  $R_{-1}(t)\in\pdo{-1}{-1-\delta}{\rr}$ and $R_{-\infty}(t)\in\pdo{-\infty}{-1-\delta}{\rr}$, we find that
$$
\slim_{t\to\pm\infty} U(0,t) R(t) U(t,0)= -\slim_{t\to\pm\infty}\int_0^t U(0,s) R_{-\infty}(s)U(s,0) ds  \in \Psi^{-\infty}(E).
$$
as claimed. \qeds

\begin{remark} It is possible to prove a generalization of  Proposition \ref{lemkey} in which  $\chi\big(H(0))$ is replaced by an arbitrary pseudo-differential operator. The necessary modifications can be carried out as in the proof of  \cite[Thm.~4.1]{Cordes1982}; this  has in fact no bearing on the large-time aspects on which we focus here.
\end{remark}

Note that because of the improvement in Sobolev order, the operators $R(t)$ in \eqref{eqitislc} are \emph{compact} in $B(H^m(\Sigma;E),H^m(\Sigma;E))$ for all $t\in \rr$. Note also that we obtain immediately an analogous statement if $\chi=\one_{\opencl{-\infty,0}}$ (instead of $\one_{\clopen{0,+\infty}}$) outside of a bounded neighborhood of $0$.

\subsection{Time-dependent scattering theory}\label{scatteringsection} Recall that by  Hypothesis \ref{hyptd} we have in particular
$$
\p_t H(t) \in \pdo{1}{-1-\delta}{\rr},
$$
where so far $\delta>0$ was allowed to be close to zero. In what follows, however, we will frequently need the short-range assumption  $\delta>1$ which  serves  to control integrability of terms involving $t \p_t H(t)$ for large $\module{t}$.


For $t,s\in\rr$ we define the operators
$$
\bea
W(t,s) &\defeq e^{-i t H(t)}U(t,s)e^{i s H(s)}.
\eea
$$
Remark that $W(t,t)=\one$ and $W(t,s)W(s,r)=W(t,r)$ for $t,s,r\in\rr$.
\bep\label{prop:sc} If $\delta>1$ then we have the existence of ``M{\o}ller wave operators'' defined as limits in the strong operator  topology, namely,
$$
\slim_{t\to \pm\infty} W(0,t)\in B(H^m(\Sigma;E))
$$
for all $m\in\rr$, and the same holds true for $W(t,0)=\big(W(0,t)\big)^{-1}$.
\eep
\proof The proof is a standard use of Cook's method. We have
$$
\p_t W(0,t)= \p_t(U(0,t)e^{itH(t)})= iU(0,t)\left(t\p_tH(t)\right)e^{itH(t)}.
$$
The r.h.s.~is $O(\bra t \ket^{-\delta})$ in $B(H^{m+1}(\Sigma;E),H^{m}(\Sigma;E))$ because $\p_tH(t)\in \pdo{1}{-1-\delta}{\rr}$ by our assumptions, and the other factors are uniformly bounded by Lemma \ref{lembd}. Therefore, for all $v\in H^{m+1}(\Sigma;E)$,
$$
\lim_{t\to\pm\infty}W(0,t)v= v + \int_0^{\pm \infty} \p_t W(0,t)  v \,dt
$$
exists in $H^{m}(\Sigma;E)$. By density of $H^{m+1}(\Sigma;E)$ in $H^{m}(\Sigma;E)$ and uniform boundedness of $W(0,t)$, we conclude that $W(0,t)$  strongly converges on the whole space $H^{m}(\Sigma;E)$.
\qeds
.

For convenience of notation, instead of writing strong operator limits we will consider   time-dependent operators as functions of $\overline{\rr}=\{ -\infty\}\cup \rr\cup\{ +\infty\}$. With this convention, $H(\pm\infty)=H_\pm$ in view of Hypothesis \ref{hyptd}.

We will use the following consequence of Proposition \ref{lemkey}.

\bel\label{compactoperator} Let $I,J\subset \rr$ be two intervals such that $I\cap J$ is bounded.  If $\delta>1$ then  for all $t,s\in\overline{\rr}$ and $m\in\rr$ the operator
\beq\label{eq:wij}
\one_I \big(H(t)\big) W( t,s ) \one_J (H(s)) \in \B{m}{m}
\eeq
is compact. Furthermore,  in the case $t=+\infty$ and  $s=-\infty$ it belongs to $\Psi^{-\infty}(E)$.
\eel
\proof If $I$ or $J$ is bounded then the assertion follows from Corollary \ref{corcor}. Otherwise, we can find $\chi_I,\chi_J\in S^0(\rr,[0,1])$ such that  $\chi_I=\one_{\clopen{0,+\infty }}$ and $\chi_J=\one_{\opencl{-\infty,0}}$  outside of a bounded neighborhood of $0$, or the other way around.  By Proposition \ref{lemkey}, for all $t\in\rr$ we have
$$
 \chi_I (H(0))   + R(0)    = U(0,t) \big( \chi_I(H(t)) + R(t)\big) U(t,0),
$$
with $R(t)\in\pdo{-1}{0}{\rr}$, and since the l.h.s.~does not depend on $t$ we conclude
$$
U(0,t) \big( \chi_I(H(t)) + R(t)\big) U(t,0)=U(0,s) \big( \chi_I(H(s)) + R(s)\big) U(s,0)
$$
for all $s\in\rr$, hence
\beq\label{chit}
 \chi_I(H(t)) U(t,s) = U(t,s)\big( \chi_I(H(s)) + R(s) \big)- R(t) U(t,s).
 \eeq
Recalling $W(t,s)=e^{-i t H(t)}U(t,s)e^{i s H(s)}$ and using that functions of $H(t)$ commute we deduce from \eqref{chit} that
\beq\label{errt}
\chi_I \big(H(t)\big) W( t,s ) = W(t,s) \chi_I\big(H(s)\big)  + W(t,s) \widetilde{R}(s)    - \widetilde{R}(t) W(t,s),
\eeq
where
$$
 \widetilde{R}(t)  = e^{-i t H(t)} R(t) e^{i t H(t)}\in \B{m}{m+1}
$$
for all $t\in \rr$. We can also write
$$
 \widetilde{R}(t) = W(t,0) U(0,t) R(t) U(t,0) W(0,t).
 $$
When $t\to\pm\infty$, the family $U(0,t) R(t) U(t,0)$  converges in the strong operator topology of $\B{m}{m+1}$ to an operator in $\Psi^{-\infty}(E)$ by  \eqref{slims} of Proposition \ref{lemkey}, and $W(t,0),W(0,t)$ preserve the mapping properties uniformly in $t$ by Proposition \ref{prop:sc}. Therefore, $\widetilde{R}(t)$ is smoothing  for infinite $t$. It follows that the two last summands in \eqref{errt} extended to $t\in\overline{\rr}$ are compact (and smoothing for infinite $t$) and can be disregarded.  By Corollary \ref{corcor} and the fact that $\chi_I \chi_J$ is compactly supported,  $\chi_I \big(H(t)\big)\chi_J \big(H(s)\big)$ is in $\Psi^{-\infty}(E)$ and thus compact, and  we conclude that the operator
$$
\chi_I \big(H(t)\big) W( t,s ) \chi_J \big(H(s)\big)
$$
is compact (and smoothing for infinite $t$). From this we can deduce that  the operator $\one_I \big(H(t)\big) W( t,s ) \one_J (H(s))$ is compact (and smoothing for infinite $t$) by using again Corollary \ref{corcor} to estimate the difference.
\qed

\subsection{Generalization for non-selfadjoint perturbations} In this paragraph we briefly comment on  generalisations to perturbations by a non-selfadjoint potential. This discussion  is auxiliary and is not needed in the next sections.

Specifically, instead of $H(t)$ we can consider operators of the form
$$
H(t)+V(t), \quad V(t)\in\pdo{0}{-\delta}{\rr},
$$
where $V(t)$ is not assumed to have any particular self-adjointness properties, but decays to zero as $t\to\pm\infty$.

Then, if we keep the definition of the principally scalar operator $\Lambda(t)=(1+H(t)^2)^\12$ unchanged, we have $[ H(t)+V(t), \Lambda^s(t)]\in \pdo{s}{-\delta}{\rr}$ for all $s\in \rr$. Furthermore, for each $t\in\rr$, $H(t)+V(t)$ is a bounded perturbation of a generator of a strongly continuous bounded semi-group, and therefore it is also the generator of a strongly continuous semi-group (see e.g.~\cite[Thm.~X.50]{RS}). In consequence, $H(t)+V(t)$ has a well-defined Schrödinger propagator $U_V(t,s)$ which satisfies  the uniform boundedness properties  from  Lemma \ref{lembd}.

Then, the proof of Proposition \ref{lemkey} still yields a statement of the form
$$
 U_V(t,0) \big(  \chi (H(0))   + R(0)  \big) U_V(0,t) =\chi(H(t)) + R(t),
$$
with $R(t)$ uniformly bounded in $\B{m}{m+1}$ for all $m\in\rr$, even though it is no longer true in general that $R(t)$ is necessarily a pseudo-differential operator in this situation. Nevertheless, if we set
$$
W_V(t,s)= e^{-it H(t)} U_V(t,s) e^{is H(s)}
$$
then the analogues of Proposition \ref{prop:sc} and Lemma \ref{compactoperator} for $W_V(t,s)$ instead of $W(t,s)$ remain valid.

In consequence, the arguments of the next section will imply the Fredholm property of the operator $\p_t - i (H(t)+V(t))$ with APS boundary conditions defined using spectral projections of $H(t)$. The inclusion of the potential $V(t)$ will   however not be needed when considering the geometric Dirac operator $D$.

\section{Fredholm inverses in the evolutionary model}\label{sec:fredholm}

\subsection{Fredholm property}\label{ss:fp} We continue the analysis of the evolutionary model introduced in Hypothesis \ref{hyptd}. Let us denote
$$
\widetilde{D}=\p_t-iH(t).
$$
The tildes will be used to stress that we consider the evolutionary model from \sec{sec:scattering} rather than the geometric Dirac operator $D$ from \sec{sec2}.

For $m\in\rr$ and small $\epsilon>0$ we define:
\beq\label{xyspaces}
\bea
\widetilde{\cY}^m&\defeq \bra t \ket^{-\12-\epsilon} L^2(\rr,H^m(\Sigma;E)),\\
\widetilde{\cX}^m&\defeq\{ u \in C^1(\rr,H^m(\Sigma;E)) \st   \widetilde{D }u \in \widetilde{\cY}^m\}.
\eea
\eeq
For $u\in \widetilde{\cX}^m$ and $t\in \overline{\rr}$ we denote
$$
\widetilde{\varrho}_\pm u\defeq  \lim_{t\to\pm \infty} e^{-itH(t)} u(t) \in H^m(\Sigma;E)
$$
provided that the limit exists. For each $u$, $\widetilde{\varrho}_+ u$ and $\widetilde{\varrho}_- u$ are interpreted as the asymptotic data of $u$, and the next proposition asserts that a solution of $\widetilde D u = f$ can be recovered from $\widetilde{\varrho}_+ u$ (as well as from $\widetilde{\varrho}_- u$).

\begin{proposition}\label{invertible}  If $\delta>1$ then for all $m\in\rr$  the operator
$$
\widetilde{\varrho}_\pm \oplus  \widetilde D : \widetilde{\cX}^m \to H^m(\Sigma;E)\oplus \widetilde{\cY}^m
$$
is well-defined, bounded and boundedly invertible.
\end{proposition}
\proof  For  $u\in\widetilde{\cX}^m$, the well-posedness of the inhomogenous Cauchy problem, see e.g.~\cite[\S5,~Thm.~5.2]{pazy}, can be expressed by the Duhamel formula
$$
u(t)=U(t,0) u(0) +\int_0^t U(t,s) f(s)ds
$$
for all $t\in\rr$, where $f=\widetilde{D}u\in \widetilde{\cY}^m$. Thus,
$$
\bea
 e^{-itH(t)}u(t)&= e^{-itH(t)} U(t,0) u(0)+\int_0^t e^{-itH(t)} U(t,0)U(0,s) f(s)ds,\\
 &=W(t,0)u(0)+\int_0^t W(t,0) U(0,s) f(s)ds,
\eea
$$
and this converges as $t\to\pm\infty$ by Proposition \ref{prop:sc}, Lemma \ref{lembd} and dominated convergence. This proves that $\widetilde{\varrho}_\pm \in B(\widetilde{\cX}^m, H^m(\Sigma;E))$ is well-defined as a strong operator limit.

Next, for any $v\in H^m(\Sigma;E)$ and $f\in \widetilde{\cY}^m$, we claim that $u=(\widetilde{\varrho}_\pm\oplus \widetilde{D})^{-1}(v,f)$ is given by the formula
\beq\label{eq:rec}
\bea
u(s)&=\lim_{t\to \pm\infty}\Big(U(s,t) e^{itH(t)}v + \int_{t}^s U(s,r) f(r)dr\Big) \\
 &= U(s,0)\lim_{t\to\pm\infty} W(0,t) v + \int_{\pm\infty}^s U(s,r) f(r)dr.
\eea
\eeq
Again, this converges by Proposition \ref{prop:sc}, Lemma \ref{lembd} and dominated convergence, and we easily find that $\widetilde{\varrho}_\pm u =v$ and $\widetilde{D} u= f$ indeed. \qeds

Next, we introduce abstract Atiyah--Patodi--Singer boundary conditions at infinity  by defining for each $m\in\rr$ the space
$$
\widetilde{\cX}^m_{\rm APS}\defeq\big\{ u \in \widetilde{\cX}^m \st \lim_{t\to+\infty}\one_{\opencl{-\infty,0}}(H(t))u(t)=0, \  \lim_{t\to-\infty}\one_{\clopen{0,+\infty}}(H(t))u(t)=0 \big\}.
$$

We denote by $\widetilde{D}_{\APS}=\widetilde{D} |_{\widetilde{\cX}^m_\APS}$ the restriction of
$\widetilde{D}$ to  $\widetilde{\cX}_{\APS}^m$.

The dependence on $m$ is not stressed explicitly  in the above notation, which is justified by the facts that for $m_2\geq m_1$,
$$\widetilde{\cX}^{m_2}_{\APS}\subset \widetilde{\cX}^{m_1}_{\APS}, \quad  \widetilde{\cY}^{m_2}\subset \widetilde{\cY}^{m_1},$$ and that the restriction of $ \widetilde{D} |_{\widetilde{\cX}^{m_1}_{\APS}}$  to  $\cX^{m_2}_\APS$  coincides with $\widetilde{D} |_{\widetilde{\cX}^{m_2}_{\APS}}$.

With all the ingredients that we already have, the proof of the Fredholm property of $\widetilde{D}_{\rm APS}$ can  now   be concluded from abstract Fredholm theory arguments (summarized in Appendix \sec{app:fredholm}) similarly as in \cite{BS}. Note that in our setting, the definition of $\widetilde{\varrho}_\pm u$ involves  a $e^{-itH(t)}$ factor, which will however not be a complication thanks to fact that it commutes with functions of $H(t)$.

\begin{proposition}\label{prop:modelcase} If $\delta>1$ then the operator $
\widetilde{D}_{\rm APS} : \widetilde{\cX}^m_{\rm APS} \to \widetilde{\cY}^m$
is Fredholm of index
$$
\ind(\widetilde{D}_{\rm APS})=\ind\big(\one_{\opencl{-\infty,0}}(H_+)  W|_{\Ran \one_{\open{-\infty,0}}(H_-)}\big).
$$
where $W=W(+\infty,-\infty)=\slim_{t\to+\infty,s\to-\infty}W(t,s)$.
\end{proposition}
\proof We apply Proposition \ref{prop:fredholm} from the appendix, with:
\beq\label{eq:dict}
\bea
\cX=\widetilde{\cX}^m, \quad  \cY=\widetilde{\cY}^m, \quad  P=\widetilde{D}_{\rm APS}: \widetilde{\cX}^m\to \widetilde{\cY}^m, \\
 \quad \varrho_-=\widetilde\varrho_-: \Ker \widetilde{D}_{\rm APS}\to H^m(\Sigma;E),\\ \varrho_+= \widetilde\varrho_+:\Ker \widetilde{D}_{\rm APS}\to H^m(\Sigma;E), \\
 \pi^+_-= \one_{\clopen{0,+\infty}}(H_-), \quad  \pi^-_+=\one_{\opencl{-\infty,0}}(H_+).
\eea
\eeq
The operator $W^{--}$ from the statement of  Proposition \ref{prop:fredholm} equals then
\beq\label{eq:qmm}
W^{--}= \one_{\opencl{-\infty,0}}(H_+)  W|_{\Ran \one_{\open{-\infty,0}}(H_-)},
\eeq
and similarly,  $W^{+-}= \one_{\open{0,+\infty}}(H_+)  W|_{\Ran \one_{\open{-\infty,0}}(H_-)}$.
Proposition \ref{prop:fredholm}  says that $W^{--}$ is Fredholm and
$$
\ind(W^{--})=\ind(\widetilde{D}|_{\Ker \varrho})
$$
where $\varrho= \pi_-^+\varrho_- \oplus \pi_+^-\varrho_+$, provided that  the following two statements are true:
\ben
\item[a)] $W^{+-}$ is compact,
\item[b)] $\varrho_\pm \oplus  P : \cX\to\cH_\pm\oplus \cY$  is boundedly invertible.
\een
In view of \eqref{eq:qmm},  a) is  in the present situation  a direct consequence  of Lemma \ref{compactoperator}. By \eqref{eq:dict}, b) follows directly  from Proposition \ref{invertible}.

It remains to prove that $\widetilde{D}_{\APS}=\widetilde{D}|_{\Ker \varrho}$, which amounts to checking that $\widetilde{\cX}^m_{\APS}= \Ker \varrho$. We have  indeed:
\begin{align*}
u\in  \widetilde{\cX}_{\APS}^m &\,\iff\, \lim_{t\to +\infty} \one_{\opencl{-\infty,0}}(H(t))u(t)=0=\lim_{t\to -\infty} \one_{\clopen{0,+\infty}}(H(t))u(t) \\
&\,\iff\, \lim_{t\to +\infty}\one_{\opencl{-\infty,0}}(H(t))e^{-it H(t)}u(t)=0=\lim_{t\to -\infty} \one_{\clopen{0,+\infty}}(H(t))e^{-it H(t)}u(t)\\
&\,\iff\,  \pi_+^-\varrho_+ u=0=\pi_-^+\varrho_- u\\
&\,\iff\,  u\in \Ker\varrho,
\end{align*}
which concludes the proof.
\qed

\subsection{Positivity properties of Fredholm inverses}\label{ss:positivity} We will now prove several properties of Fredholm inverses of $\widetilde{D}_{\rm APS}$ that will be useful in analysing microlocal properties of Fredholm inverses of $\widetilde{D}_{\rm APS}$.

Let us first introduce the retarded/advanced inverse of $\widetilde{D}$,  defined by
$$
\big(\widetilde D_+^{-1} f\big) (t) \defeq \int_{-\infty}^t U(t,s) f(s) \,ds, \quad \big(\widetilde D_-^{-1} f\big) (t) \defeq \int_{+\infty}^{t} U(t,s) f(s) \,ds
$$
for all $f\in \widetilde \cY^m$. Then, $\widetilde{D}_\pm^{-1}:\widetilde\cY^m\to \widetilde\cX^m$ is bounded $\widetilde{D}\circ \widetilde{D}_\pm^{-1}=\one$, and from \eqref{eq:rec} we can conclude that
$$
\widetilde{D}_\pm^{-1} =(\widetilde\varrho_\mp\oplus \widetilde D)^{-1}(0\oplus\one).
$$

Recall that for  $t\in \overline{\rr}$, $H(t)$ is self-adjoint w.r.t.~the $t$-dependent {\Red $L^2_t(\Sigma;E)$ scalar product with norm $( \cdot | \cdot)_{L^2_t}=( \cm(t) \cdot  | \cm(t)\, \cdot)_{L^2}$.}
We also denote by $( \cdot | \cdot)_{L^2_t}$ the $L^2_t(\Sigma;E)$-pairing of $H^m(\Sigma;E)$ and $H^{-m}(\Sigma;E)$, and abbreviate it by  $( \cdot | \cdot)_{L^2_\pm}$ in the case $t=\pm\infty$.

For square-integrable functions $f_1$ and $f_2$ with values in respectively $H^m(\Sigma;E)$ and $H^{-m}(\Sigma;E)$, we denote
$$
(f_1| f_2)= \int_{\rr} ( f_1(t) | f_2(t))_{L^2_t} \,dt.
$$

\begin{lemma}\label{sff} Let $\widetilde\varrho_\pm^{-1}: H^m(\Sigma;E)\to \Ker \widetilde D$ be defined by $\widetilde\varrho_\pm^{-1}=(\widetilde\varrho_\mp\oplus \widetilde D)^{-1}(\one\oplus 0)$. Then $\widetilde\varrho_\pm^{-1}$ is the inverse of $\widetilde \varrho_\pm|_{\Ker \widetilde{D}}$, and it satisfies
$$
(\widetilde\varrho_\pm \widetilde D_\pm^{-1} f | h )_{L^2_\pm} = \pm ( f | \widetilde\varrho_\pm^{-1} h)
$$
for all $f\in \widetilde\cY^{-m}$ and $h\in H^{m}(\Sigma;E)$.
\end{lemma}
\proof  On the one hand, $\big(\widetilde D_\pm^{-1} f\big) (t) = \int_{\mp\infty}^t U(t,s) f(s) \,ds
$,
and thus
$$
\widetilde \varrho_\pm \widetilde D_\pm^{-1}  f   = \lim_{t\to\pm\infty}  \int_{\mp\infty}^{\pm\infty} W(t,0)  U(0,s) f(s) \,ds.
$$
On the other hand, $\big(\widetilde\varrho_\pm^{-1}  h\big)(s)=\lim_{t\to\pm\infty}U(s,0)W(0,t)h$. In consequence,
$$
\bea
(\widetilde\varrho_\pm \widetilde D_\pm^{-1} f | h )_{L^2_{\pm}} &= \lim_{t\to\pm\infty}  \int_{\mp\infty}^{\pm\infty} (W(t,0)  U(0,s) f(s)|h)_{L^2_t} \,ds  \\
&= \lim_{t\to\pm\infty}  \int_{\mp\infty}^{\pm\infty} ( f(s)| U(s,0)W(0,t)h)_{L^2_s} \,ds = \pm ( f | \widetilde\varrho_\pm^{-1} h)
\eea
$$
as claimed, where to go from the first line to the second we used the fact that $U(t,s)$ and $W(t,s): {\Red L^2_s}(\Sigma;E)\to  {\Red L^2_t}(\Sigma;E)$ are unitary for all $t,s\in\rr$. The latter property of $U(t,s)$ can be shown using a standard positive energy estimate argument, see e.g.~\cite[Lem.~2.4]{BS}, and then the case of $W(t,s)$  follows easily.  \qeds

 Let now $\widetilde\cK^m\subset\widetilde\cX^m $  and $\widetilde\cR^m\subset\widetilde\cY^m$ be closed subspaces such  that
 $$
 \widetilde\cX_\APS^m=\Ker \widetilde D_{\APS} \oplus \widetilde\cK^m, \quad \widetilde\cY^m=\Ran \widetilde D_{\APS} \oplus \widetilde\cR^m,
 $$
 and $\dim \widetilde\cR^m<+\infty$. Let  $\widetilde{D}_\APS^\inv$ be the associated Fredholm inverse\footnote{If $P:\cX\to \cY$ is a Fredholm operator acting between  two Banach spaces $\cX$, $\cY$, a \emph{Fredholm inverse} of $P$ is a bounded operator $P^{\inv}:\cY\to\cX$ such that $P\circ P^{\inv}=\one_\cY$ modulo a compact operator and $P^{\inv}\circ P=\one_\cX$ modulo a compact operator. We avoid using the term ``parametrix'' in this context to avoid confusion with parametrices in the sense of smooth regularity.} of $\widetilde{D}_{\rm APS} : \widetilde{\cX}^m_{\rm APS} \to \widetilde{\cY}^m$,  uniquely defined by the property that
 \beq\label{ptl2}
 \widetilde{D}\circ \widetilde{D}_\APS^\inv = \one - R, \quad \widetilde{D}_\APS^\inv\circ \widetilde{D} = \one - L,
 \eeq
 where $R$ is the projection to $\widetilde{\cR}^m$ along $\Ran \widetilde{D}_{\APS}$ and $L$ the projection to $\Ker \widetilde{D}_{\APS}$ along  $\widetilde{\cK}^m$.

For different $m\in\rr$ we can choose the complementary subspaces in a compatible way, in the sense that
\beq\label{compatibles}
\widetilde\cK^{m_2}\subset \widetilde\cK^{m_1}, \quad    \widetilde\cR^{m_2}\subset\widetilde \cR^{m_1}
\eeq
 for $m_2\geq m_1$. Furthermore, by density of $\cf(M;S^+M)$ in $\cY^m$ we can choose the finite dimensional space $\widetilde{\cR}^m$ in such way that $\widetilde{\cR}^m\subset \cf(M;S^+M)$.

The $m$-independent  notation is justified by the compatibility inclusions \eqref{compatibles}.

\begin{proposition}\label{prop:positivity} Assume $\delta>1$. If $\widetilde D^{\inv}_{\APS}$ is a Fredholm inverse as above, then there exists a smoothing operator $K_\pm\in \bigcap_{m,s\in \rr} B(\widetilde{\cY}^m,\widetilde{\cX}^{m+s} ) $ such that for all $f\in \widetilde\cY^0$,
$$
( f  | (\widetilde D^{\inv}_{\APS} - \widetilde D_\pm^{-1} +K_\pm)  f  )\geq 0.
$$
\end{proposition}
\proof We focus on  $\widetilde{D}_-^{-1}$, the other case  being analogous. Let $Q:  \widetilde\cY^m\to\widetilde\cX^m$ be given by $
Q =(\one  - \varrho_-^{-1} \pi_-^+ \varrho_-)\widetilde D_-^{-1}$ in the notation introduced in \eqref{eq:dict}. Then, by Proposition \ref{propq}, the difference
$$
K_-\defeq Q-D^{\inv}_{\APS} : \widetilde\cY^m\to\widetilde\cX^m
$$
is compact. A close inspection of  formulae  \eqref{K1} and \eqref{K2} in the proof of  Proposition \ref{propq} shows that $K_-=E_1+E_2$, where $E_1$ is the composition of
$$
W^{-+} =\one_{\opencl{-\infty,0}}(H_+)  W|_{\Ran \one_{\clopen{0,+\infty}}(H_-)}.
$$
 with operators  that preserve regularity, and $E_2$ satisfies $\Ran E_2\subset \Ker \widetilde{D}_{\rm APS}$. By   Lemma \ref{compactoperator},  $W^{--}\in \Psi^{-\infty}(E)$, therefore $E_1$ is smoothing. Furthermore, elements of $\Ker \widetilde{D}_{\rm APS}$ are smooth by an argument which  is postponed for the moment, see Proposition \ref{propwf} and Corollary \ref{coroc}; hence, $E_2$ is smoothing as well. We conclude that  $K_-$ is smoothing, and consequently it suffices to prove positivity of  $Q-\widetilde D_-^{-1}$.

 We have $Q-\widetilde D_-^{-1}=-\varrho_-^{-1} \pi_-^+ \varrho_-\widetilde D_-^{-1}$, and by Lemma \ref{sff},
 $$
 \bea
 -(f|\varrho_-^{-1} \pi_-^+ \varrho_-\widetilde D_-^{-1}f)& =(\varrho_-\widetilde D_-^{-1}f |  \pi_-^+ \varrho_-\widetilde D_-^{-1}f)_{L^2_-} \\ &= (\varrho_-\widetilde D_-^{-1}f |  \one_{\clopen{0,+\infty}}(H_-) \varrho_-\widetilde D_-^{-1}f)_{L^2_-}\geq 0
 \eea
 $$
for all $f\in \widetilde\cY^0$ as asserted. \qed

\section{Index theory}\label{ss:index}

\subsection{The spectral flow} Once the Fredholm property is shown, we will employ a spectral flow argument largely analogous to the work of Bär--Strohmaier \cite{BS} to compute the index.  

We first introduce the relevant terminology; see \cite{Lesch2004,phi,BS,dun3,Drouot2019} for more details  and various applications.

Let $\cH$ be a separable Hilbert space and let $\{B(t)\}_{t\in I}$ be a {\Red norm-continuous} family of self-adjoint Fredholm operators on $\cH$.

\bed\label{spectralprojection}
For an interval $J\subset\rr$, we denote
$$
P_J(t)\defeq \one_J(B(t)),\qquad \cH_J(t)\defeq \Ran (B_J(t)).
$$
For $a\in\rr$, we will simply write
$$
P_{<a}(t)\defeq P_{\open{-\infty,a}}(t),\qquad \cH_{<a}(t) \defeq \Ran (P_{<a}(t)),
$$
and similarly for $\geq a$.
\eed

We recall the following definition of the spectral flow from \cite{phi}.
\bed\label{spectralflow}
For a compact interval $I=[t_1,t_2]$, a partition
$$
t_1=\tau_0<\tau_1<...<\tau_N=t_2
$$
together with numbers $a_j \in \rr$ for $0\leq j\leq N$ is called a \emph{flow partition} (for $B(\cdot)$) if for each $n$ and $t\in[\tau_{n-1},\tau_n]$ we have $a_n\notin \spe (B(t))$ and $\cH_{\clopen{0,a_n}}(t)$ is finite dimensional. For such a partition, the \emph{spectral flow} is defined as
$$
\sfl_I(B)= \sum_{n=1}^N \dim (\cH_{\clopen{0,a_n}}(\tau_n)) -\dim (\cH_{\clopen{0,a_n}}(\tau_{n-1}) ).
$$
\eed
The spectral flow is well-defined, i.e., a flow partition exists and the spectral flow is independent of the choice of flow partition, see \cite{phi} for more details.

Since in our main case of interest we are dealing with families  $\{B(t)\}_{t\in I}$ parametrized by $I=\overline{\rr}$, it is convenient to define
\beq\label{convsf}
\sf_{\overline{\rr}}(B)\defeq \sf_{[-1,1]}(\overline{B})
\eeq
where $\overline{B}(t)=B(\phi(t))$ and $\phi :[-1,1] \to \overline{\rr}$ is an  arbitrarily chosen  homeomorphism, and we apply the analogous convention for other infinite intervals.

\subsection{The index and the spectral flow}

Suppose that $H(t)\in\pdo{1}{0}{\rr}$ satisfies Hypothesis \ref{hyptd}, in particular $H(t)=\cp (t) H_0(t) \cm (t)$ with $H_0(t)$ elliptic and self-adjoint, and thus Fredholm. We apply the definitions  introduced in the previous paragraph to the family $H_0(t)$ and the Hilbert space $\cH=L^2(\Sigma;E)$. In particular, the spectral flow $\sfl_{\overline{\rr}}(H_0)$ is well-defined.

 Recall that the family of operators $W(t,s)$ was defined for $t,s\in \rr$ in  \sec{scatteringsection}   and  then extended to $t,s\in\overline{\rr}$.  If we now denote
$$
W_0(t,s)\defeq \cp (t) W(t,s) \cm (s),
$$
then $W_0(t,s)$ has properties analogous to $W(t,s)$, in particular
$$
W_0(t,r)W_0(r,s)=W_0(t,s) \mbox{ for all }t,r,s\in\overline{\rr}.
$$
Furthermore,  $W_0(t,s)$ is bounded in $\B{m}{m}$ for all $m\in\rr$ uniformly in $t,s\in \overline{\rr}$ by  Lemma \ref{lembd} and uniform boundedness of $\cp (t),\cm (t)$.

Next, we introduce the notation
$$
W^{--}_0(t,s)\defeq P_{<0}(t) \circ W_0(t,s)|_{\cH_{<0}(s)}, \quad W^{+-}_0(t,s)\defeq P_{\geq 0}(t) \circ W_0(t,s)|_{\cH_{<0}(s)},
$$
(and similarly for $W^{++}_0(t,s)$ and $W^{-+}_0(t,s)$) for the components of $W_0(t,s)$ relative to the two decompositions $\cH=\cH_{<0}(s)\oplus \cH_{\geq 0}(s)$ and $\cH=\cH_{<0}(t)\oplus \cH_{\geq 0}(t)$.

We show the following result which equates the spectral flow with an index, and which plays the role of the analogue of \cite[Thm.~4.1]{BS} in our setting.
\bep\label{indexflow}
If $\delta>1$ then for each $t,s\in \overline{\rr}$ with $t\geq s$, $W^{--}_0(t,s)$ is Fredholm and satisfies
\beq
\ind(W^{--}_0(t,s))=\sfl_{[s,t]}(H_0).
\eeq
\eep
\proof \step{1} In view of the relation $H(t)=\cp (t) H_0(t) \cm (t)$, Lemma \ref{compactoperator} implies that $W^{+-}_0(t,s)$ is compact for each $t,s\in\overline{\rr}$. This implies that $W^{--}_0(t,s)$ is Fredholm by a simple argument recalled in Proposition \ref{prop:fredholm} in the appendix.

\step{2} From this point on, the arguments are fully analogous to  the proof of \cite[Thm.~4.1]{BS}  thanks to the  properties of $W_0(t,s)$ which we already showed. For the reader's convenience we repeat these arguments below (see also \cite{dun3} for an alternative approach). The cases of $t,s$  finite and infinite will be taken care of simultaneously, in accordance with the notation \eqref{convsf}.

By Definition \ref{spectralflow}, we can choose a partition $s=\tau_0<\tau_1<...<\tau_N=t$ and numbers $a_j$ such that $\pm a_j\notin \spe(H_0(\tau))$ for all $\tau\in[\tau_{j-1},\tau_j]$.
\beq
\sfl_{[s,t]}(H_0)=\sum_{j=1}^N \left(\dim \cH_{\clopen{0,a_j}}(\tau_j)-\dim \cH_{\clopen{0,a_j}}(\tau_{j-1}) \right).\label{bs20}
\eeq
Since $a_j \notin \spe(H_0(\tau))$, the family of projectors $P_{<a_j}(\tau)$ is continuous on $\cH$ over $[\tau_{j-1},\tau_{j}]$. By the same arguments as in \emph{Step 1} we can show that
$$
P_{<a_j}(\tau)\circ W_0(\tau,s):\cH_{<0}(s)\to\cH_{<a_j}(\tau)
$$
is a continuous family of Fredholm operators for $\tau\in [\tau_{j-1},\tau_j]$.  By \cite[Lem.~3.2]{lesch}, we obtain
\beq
\ind(P_{<a_j}(\tau_j)\circ W_0(\tau_j,s)) = \ind(P_{<a_j}(\tau_{j-1})\circ W_0(\tau_{j-1},s)).\label{bs21}
\eeq
If we consider both operators $P_{<a_j}(\tau)\circ W_0(\tau,s)$ and $P_{<0}(\tau)\circ W_0(\tau,s)$ as operators $\cH_{<0}(s)\to \cH_{<a_j}(\tau)$, then they differ by $P_{\clopen{0,a_j}}(\tau)\circ W_0(\tau,s)$, which is a compact operator by Lemma \ref{compactoperator} (applied to $I=\clopen{0,a_j}$, $J=\open{s,0}$) combined with the identity $H=\cp  H_0 \cm $. Therefore,
\beq
\ind(P_{<a_j}(\tau)\circ W_0(\tau,s))=\ind(P_{<0}(\tau)\circ W_0(\tau,s)), \label{bs22}
\eeq
where both operators are considered as operators from $\cH_{<0}(s)$ to $\cH_{<a_j}(\tau)$.

Now, $W^{--}_0(\tau,s)$ coincides with $P_{<0}(\tau)\circ W_0(\tau,s)$, with the difference that it is considered as an operator to $\cH_{<0}(\tau)$. Hence
\beq
\ind(W^{--}_0(\tau,s))=\ind(P_{<0}(\tau)\circ W_0(\tau,s)) + \dim\cH_{\clopen{0,a_j}}(\tau),\quad \forall \tau\in [\tau_{j-1},\tau_j].\label{bs23}
\eeq
We conclude:
\begin{align*}
\ind(W^{--}_0(t,s))=&\sum_{j=1}^N \left( \ind(W^{--}_0(\tau_j,s))-\ind[W^{--}_0(\tau_{j-1},s)] \right) \\
=&\sum_{j=1}^N \big(\ind(P_{<0}(\tau_j)\circ W_0(\tau_j,s)) + \dim\cH_{\clopen{0,a_j}}(\tau_j)\\
&-\ind(P_{<0}(\tau_{j-1})\circ W_0(\tau_{j-1},s)- \dim\cH_{\clopen{0,a_j}}(\tau_{j-1})\big) \\
=&\sum_{j=1}^N \big(\ind(P_{<0}(\tau_j)\circ W_0(\tau_j,s))-\ind(P_{<0}(\tau_{j-1})\circ W_0(\tau_{j-1},s))\big)\\ &+\sfl_{[s,t]} (H_0)\\
=&\sum_{j=1}^N \big(\ind(P_{<a_j}(\tau_j)\circ W_0(\tau_j,s)\big)-\ind\big(P_{<a_j}(\tau_{j-1})\circ W_0(\tau_{j-1},s))\big)\\ &+\sfl_{[s,t]}(H_0)\\
=&\,\sfl_{[s,t]} (H_0),
\end{align*}
where we used \eqref{bs23}, \eqref{bs20}, \eqref{bs22} and \eqref{bs21} from the second to last step.
\qed

\begin{remark} Alternatively, one could carry out the same analysis for $H(t)$ and $W(t,s)$ directly (rather than for $H_0(t)$ and $W_0(t,s)$), at the slight cost of having to work with operators similar to self-adjoint ones, or with $t$-dependent Hilbert spaces as in \cite{BS}.
\end{remark}

\subsection{Index theorem for the Dirac operator}\label{ss:ind} We now consider the setting of the geometric Dirac operator introduced in \secs{ss1}{ss3}, and  we use the relationship with the evolutionary model to deduce the index theorem.

  Recall  in particular that our main operator of interest is
\beq\label{eq:defD}
D= -\nabla_n^{\SM}-i {A}(t)-r(t) :  \cf(M;S^+ M)\to\cf(M; S^+M).
\eeq
For small $\epsilon>0$, we define the spaces:
\beq\label{xyspaces2}
\bea
\cY&\defeq \bra t \ket^{-\12-\epsilon} L^2_t L^2_y(M;\SpM)\\
\cX&\defeq\{ u \in C^0_t L^2_y(M;\SpM) \st  Du \in \cY\},
\eea
\eeq
and we introduce boundary condition at infinity as follows.

\begin{definition} For $\cX$ as in \eqref{xyspaces2}, the subspace of functions satisfying \emph{Atiyah--Patodi--Singer conditions at infinity} is defined as
$$
\cX_{\rm APS}\defeq\big\{ u \in \cX \st \lim_{t\to+\infty} \one_{\opencl{-\infty,0}}(A(t))u(t)=0, \ \lim_{t\to-\infty} \one_{\clopen{0,+\infty}}(A(t))u(t)=0 \big\}.
$$
We denote by $D_{\rm APS}$ the restriction of $D$ to $\cX_{\rm APS}$.
\end{definition}


\begin{theorem}\label{thm:final} Assume $(M,g)$ is a Lorentzian spacetime equipped with a spin structure, such that $M=\rr\times \Sigma$ for some $\Sigma$ compact and odd-dimensional, and such that $\Sigma_t=\{t\}\times \Sigma$ is a Cauchy surface for each $t\in\rr$. Suppose that the metric $g$ satisfies Hypothesis \ref{hypothesis1} with $\delta>1$, and  let $D_{\rm APS}: \cX_{\rm APS} \to \cY$ be defined as above. Then, $D_{\rm APS}$ is Fredholm of index
\beq\label{theindex}
\ind(D_{\rm APS})=\int_M \widehat{\rm A}+ \int_{\pMs} {\mathrm T}\widehat{\rm A}+\12\big(\eta(A_+,A_-) -\dim\ker(A_+)-\dim\ker(A_-)\big),
 \eeq
 where $\eta(A_+,A_-)=\eta(A_+)-\eta(A_-)$ is the difference of the eta forms of $A_+$ and $A_-$.
\end{theorem}

Above, $\Aform$ is the \emph{Atiyah--Singer integrand} (or $\Aform$-form), associated with the Levi--Civita connection $\nabla$ on $(M,g)$ (see e.g.~\cite[\S10.5--\S10.6]{Taylor2011} for an introductory account). The boundary integral involves the \emph{transgression form} ${\mathrm T}\widehat{\rm A}$ of $(M,g)$, which is  defined in terms of $\nabla$ and a reference connection for an auxiliary Riemannian metric, though its pullback to the boundary does not depend on the choice of the latter, see \cite[\S4]{BS}.

\begin{remark} Theorem \ref{thm:final} is valid regardless of the precise choice of decaying weight in the definition of the space $\cY$; one can also choose a weight that depends on spatial variables.
\end{remark}

In the proof of Theorem \ref{thm:final}  we will use the relationship between $D$ and the evolutionary  setting considered in  \sec{sec:scattering}. Recall that in the setting of Hypothesis \ref{hypothesis1},  the metric $g$ is assumed  to be of the form $
g=- c^2(t)dt^2 + h_{ij}(t)dy^i dy^j$ with smooth $c>0$. By Lemma \ref{lem:equivalence}, the operator $D$ satisfies
\beq\label{eq:equivalence2}
\bea
D&= U(t)^{-1}c^{-1}(t) \circ \widetilde{D} \circ U(t) \\ &= U(t)^{-1}c^{-1}(t) \big(\p_t - i H(t)\big) U(t),
\eea
\eeq
where  $H(t)=  c(t) U(t) A(t) {U(t)}^{-1}$, $A(t)$ is the Dirac operator on  $(\Sigma_t,h(t))$ and $U(t)$ was defined in $\sec{ss:foliation}$. Let us also recall that Hypothesis \ref{hypothesis1} states that:
\begin{enumerate}
\item\label{thyp1p} $c(t)-c_\pm\in S^{-\delta}(\rr_\pm,\cf(\Sigma))$ for some $c_\pm\in \cf(\Sigma)$ s.t.~$c_\pm>0$,
\item\label{thyp2p}  $h(t)-h_\pm\in S^{-\delta}(\rr_\pm,\cf(T^*\Sigma\otimes_{\rm s}T^*\Sigma))$ for some Riemannian metric $h_\pm$.
\end{enumerate}

\begin{lemma}\label{lemrel} Under the assumptions of Theorem \ref{thm:final}, the family of first order differential operators $H(t)$ satisfies Hypothesis \ref{hyptd}, in particular $H(t)\in \pdos{1}{0}{\rr}$ and $H(t)-H_\pm\in \pdos{1}{-1-\delta}{\rr_\pm}$ for some $H_\pm\in \Psi^1(S\Sigma)$ of Dirac type.
\end{lemma}
\begin{proof} {\Red For each $t\in\rr$, by polar decomposition and invertibility $U(t)=\module{U^*(t)} U_1(t)$ for some unitary operator $U_1(t)$. From the  definition  $H(t)= c(t) U(t) A(t) {U(t)}^{-1}$ we obtain immediately $H(t)=\cp (t) H_0(t) \cm (t)$ where}
$$
\Red H_0(t) = U_1(t) c^\12(t)  A(t) c^\12 (t) U_1(t)^{-1}, \quad \cp(t)=c^\12(t)\module{U^*(t)}.
$$
We first show that $H_0(t)\in  \pdos{1}{0}{\rr}$. This follows from the fact that $U(t)$ and the coefficients of $A(t)$ depend on time only through $h(t)$ and its derivatives, which all behave as decaying symbols by \eqref{thyp2}. More precisely, recall that $U(t)=\rho(t)\tau_t=|h(0)|^{-\frac{1}{4}}|h(t)|^{\frac{1}{4}}\tau_t$, and the dependence on $t$ of the parallel transport $\tau_t$ can be deduced from the corresponding system of ODEs, see \cite[\S5]{BGM}. 

{\Red  Next, let us check that $\cp(t), \cm(t)\in  \pdo{0}{0}{\rr}$ and  $\p_t \cp(t)\in  \pdo{0}{-1-\delta}{\rr}$.  First, by  \eqref{thyp1} of Hypothesis \ref{hypothesis1} recalled above, $c(t)$ satisfies  $c(t)\in S^0(\rr,C^\infty(\Sigma))$, $\p_t c(t)\in S^{-1-\delta}(\rr,C^\infty(\Sigma))$ and $c_0^{-1}\leq c(t) \leq c_0$  uniformly in  $t\in \rr$ for some $c_0>1$, so analogous properties are satisfied by $c^\12(t)$. Furthermore, we can compute $U^*(t)$ as in \cite[\S3.2]{dun1}, and we find that 
$$
\module{U^*(t)}^2=U(t)U(t)^*=(\tau_t^0 \gamma_t \tau_0^t) \gamma_0\in S^{0}(\rr, C^\infty(\Sigma;\End(S\Sigma)))
$$
and $\p_t \module{U^*(t)}^2 \in S^{-1-\delta}(\rr, C^\infty(\Sigma;\End(S\Sigma)))$  using Hypothesis \ref{hypothesis1}. This implies $\module{U^*(t)} \in  \pdo{0}{0}{\rr}$  and  $\p_t \module{U^*(t)}\in  \pdo{0}{-1-\delta}{\rr}$, hence the stated properties of $T(t)$. 

Since $A(t)$ is elliptic, formally self-adjoint and of Dirac type,  the same is true of $H_0(t)$. Finally, the statements on $H(t)$ are concluded from the proprieties of $H_0(t)$ and $T(t)$ shown above.}
\end{proof}

\begin{refproof}{Theorem \ref{thm:final}} The proof is split in two parts: we first show the Fredholm property using the results from \sec{sec:scattering},   and then the index is computed using a straightforward extension of the method from \cite{BS}.

 \step{1} Recall that in \sec{ss3} we introduced an isomorphism $U  :   C^0_t L^2_y(M;\SpM)\to C^0(\rr, L^2(\Sigma;S\Sigma))$, and it is easily seen to extend to an isomorphism $U\cY= \widetilde{\cY}^0$. By Lemma \ref{lem:equivalence}, $
D= U^{-1}c(t)^{-1} \widetilde{D} U$,  hence  we also have $U\cX=\widetilde{\cX}^0$.

It straightforward to check that boundary conditions are mapped consistently by $U$, i.e.~$U \cX= \widetilde{\cX}^0$, and consequently
\beq\label{eq:daps}
D_{\rm APS}= U^{-1}c^{-1} \widetilde{D}_{\rm APS} U: \cX_{\rm APS}\to \cY.
\eeq
By Proposition \ref{prop:modelcase}, $\widetilde{D}_{\rm APS}$ is Fredholm, and therefore so is $D_{\rm APS}$, with $\ind({D}_{\rm APS})=\ind(  \widetilde{D}_{\rm APS})$. By Proposition \ref{prop:modelcase} combined with Proposition \ref{indexflow}, we have
\beq
\bea
\label{st1}
 \ind(D_{\APS})&=\ind \big(\one_{\opencl{-\infty,0}}(H_+)W\big|_{\Ran \one_{\open{-\infty,0}}(H_-)}\big)\\
 &=\ind(W_0^{--}(+\infty,-\infty))-\dim\Ker H_+\\
 &=\sf_{\overline{\rr}}(H_0)-\dim\Ker A_+\\
 &=\sf_{\overline{\rr}}(A)-\dim\Ker A_+,
 \eea
\eeq
provided that we justify the last identity. To show this we write  $H_0(t)$  in the form {\Red $H_0(t)=c^\12(t) U_1(t)   A(t)  U_1(t)^{-1} c^\12 (t)$}, and since unitary transformations preserve the spectral flow (extended in the natural way to the Hilbert space bundle setting), see e.g.~\cite[Thm.~3.14]{ronge}, it suffices to know that $c^{-\12} (t)H_0(t) c^{-\12} (t)$ and $H_0(t)$ have equal spectral flow. This in turn can be proved by considering the continuous deformation $$[0,1]\ni s \mapsto c^{-\frac{s}{2}}(t)H_0(t) c^{-\frac{s}{2}}(t)$$ and by using e.g.~\cite[Prop.~3]{phi} to equate the spectral flows for $s=0$ and $s=1$.

\step{2} We can find a smooth Riemannian metric $g^{\Eucl}$ on $\rr\times \Sigma$ such that  ${g}^{\Eucl}=  dt^2+h_+$ for all $t>\12$ and ${g}^{\Eucl}=  dt^2+h_-$ for all $t<-\12$. Let ${D}^{\Eucl}$ be the Riemannian Dirac operator associated to ${g}^{\Eucl}$ and let ${A}^{\Eucl}(t)$ be the induced Dirac operator on $\{t\}\times \Sigma$, in particular ${A}^{\Eucl}(t)=A_\pm(t)$ for $\pm t>\pm \12$. Furthermore, let $T>\12$  and let ${D}^{\Eucl}_{\rm APS}$ be the restriction of ${D}^{\Eucl}$ to the space of $L^2$ functions on $[-T,T]$ with Atiyah--Patodi--Singer boundary conditions at $-T$ and $T$ (defined using respectively ${A}^{\Eucl}(-T)=A_-$ and ${A}^{\Eucl}(T)=A_+$). By  the relationship between the index of ${D}^{\Eucl}_{\rm APS}$ and the spectral flow, see e.g.~\cite[(18)]{BS}, we have
\beq\label{fin1}
\bea
\ind D^{\Eucl}_{\rm APS}&= \sf_{[-1,1]}(\overline{A})-\dim\Ker A_+
\eea
\eeq
using that $\overline{A}(\pm 1)=A_\pm$.  Recall that by definition  $ \sf_{\overline{\rr}}(A)=\sf_{[-1,1]}(\overline{A})$, so in view of \eqref{st1} we get
\beq\label{fin2}
\ind D^{\Eucl}_{\rm APS}= \ind D_{\rm APS}.
\eeq
We apply the Atiyah--Patodi--Singer theorem \cite[Thm.~3.10]{Atiyah1975a}, which gives
\beq\label{fin3}
\ind D^{\Eucl}_{\rm APS}= \int_{[-T,T]\times \Sigma} \widehat{\rm A}^{\Eucl}+\12\big(\eta(A_+,A_-) -\dim\ker(A_+)-\dim\ker(A_-)\big)
\eeq
On the other hand,
\beq\label{fin4}
\int_{[-T,T]\times \Sigma} \widehat{\rm A}^{\Eucl}  = \int_{[-T,T]\times \Sigma}  \widehat{\rm A} + \int_{\p [-T,T]\times \Sigma} {\mathrm T}\widehat{\rm A}.
\eeq
As long as $T>\12$, the l.h.s.~does not depend on $T$  in view of \eqref{fin3} and \eqref{fin1}. Using Hypothesis \ref{hypothesis1}, we control  the convergence of the boundary term on the r.h.s., as it is the integral of a polynomial in the curvature tensor, the second fundamental form of the boundary and some derivatives of $c(t)$. Therefore, we can take the $T\to +\infty$ limit in \eqref{fin4} and \eqref{fin3}, which combined with \eqref{fin2} yields  \eqref{theindex}.
\end{refproof}

\section{Microlocal properties of Fredholm inverses}\label{sec:feynman}

\subsection{Preliminaries on wavefront sets} In this final section we analyse microlocal properties of Fredholm inverses of $D_{\rm APS}$.

Let us start by introducing the necessary background on wavefront sets, see e.g.~\cite{H,HormanderIII} or \cite[\S7]{G} for more details.

Let $\zero$ denote the zero section of $T^*M$. Recall that for each $u\in H^m_{\rm loc}(M;S^+M)$ with $m\in\rr$, its wavefront set $\wf(u)$ is the subset of $T^*M\setminus\zero$  defined as follows: $q\in T^*M\setminus \zero$ does \emph{not} belong to $\wf(u)$ if and only if there exists a properly supported pseudo-differential operator $A\in \Psi^0(M;S^+M)$ such that $Au\in \cf(M;S^+M)$. If $G:C_{\rm c}^\infty(M;S^+M)\to  C^\infty(M;S^+M)$ is a continuous operator then $\wf(G)$ is by definition the wavefront set of the Schwartz kernel of $G$. One also uses the somewhat more convenient primed wavefront set $\wf'(G)$, defined by
$$
(q_1,q_2)\in\wf'(G) \,  \iff  \, (q_1,-q_2) \in \wf(G),
$$
where the notation $-q_2$ refers to multiplication by $-1$ in the fibers.

In the geometric context introduced in \sec{sec2}, the characteristic set of $D$ is the following subset of $T^*M\setminus \zero$:
$$
\cN=\{ q=(x,\xi)\in T^*M\setminus\zero \st  \xi\cdot g^{-1}(x)\xi = 0 \},
$$
i.e.~$\cN$ is the zero set of the principal symbol $p(x,\xi)=\xi\cdot g^{-1}(x)\xi$ of the Lorentzian Laplace--Beltrami operator on $(M,g)$.
Integral curves in $\cN$ of the (forward, backward) Hamilton flow of $p$ are called (forward, backward) \emph{null bicharacteristics}. The characteristic set $\cN$  has two connected components which  can be distinguished one from the other by setting
$$
\cN^{\pm}\defeq\cN\cap \{(x, \xi)\in T^*M\setminus\zero \st  \, \forall v\in T_{x}M\hbox{ future directed time-like},\, \pm v\dual \xi>0 \}.
$$

\begin{definition}\label{deffeynman} One says that  $G:C_{\rm c}^\infty(M;S^+M)\to  C^\infty(M;S^+M)$ has \emph{Feynman wavefront set} if
\beq\label{tobef}
\wf'(G)\setminus {T^*_\Delta (M\times M)} \subset \{ (q_1, q_2) \in  \cN \times \cN \st q_1 \succ q_2 \}
\eeq
where $q_1 \succ q_2$ means that $q_1$ can be reached from $q_2$ by a forward null bicharacteristic, and $T^*_\Delta (M\times M)=\{ (q,q) \st q\in T^*M \}$ is the diagonal in $T^*M\times T^*M$.
\end{definition}

\begin{remark}\label{remrem} If $q_1=(x_1,\xi_1)\in \cN^+$ (resp.~$\cN^-$) and  $q_2=(x_2,\xi_2)\in\cN^+$ (resp.~$\cN^-$) are  on the same null bicharacteristic then $q_1 \succ q_2$ if and only if $x_1$ is in the causal future (resp.~past) of $x_2$ .
\end{remark}

\subsection{Wavefront set of parametrices of $D_\APS$}  It is useful to generalize the definitions  of the spaces $\cX,\cY$ and related objects from \sec{ss:ind} by setting for $m\in \rr$,
$$
\bea
\cY^m&\defeq \bra t \ket^{-\12-\epsilon} L^2_t H^m_y(M;\SpM),\\
\cX^m&\defeq\{ u \in C^0_t H^m_y(M;\SpM) \st  Du \in \cY^m\},
\eea
$$
and
$$
\cX^m_{\rm APS}\defeq\big\{ u \in \cX^m \st \lim_{t\to+\infty} \one_{\opencl{-\infty,0}}(A(t))u(t)=0, \ \lim_{t\to-\infty} \one_{\clopen{0,+\infty}}(A(t))u(t)=0 \big\}.
$$
Since $\cY^{m_2}\subset \cY^{m_1}$ and $\cX^{m_2}_{\APS}\subset \cX^{m_1}_{\APS}$ for $m_2\geq m_1$ it makes sense to write $D_{\rm APS}: \cX^m_{\rm APS}\to \cY^m$ for the restriction of $D$ to $\cX^m_{\rm APS}$.

The key argument is summarized in the next proposition.

\begin{proposition}\label{propwf} Let $m\in\rr$, suppose that $u\in \cX^m$ satisfies $Du =f$ with $f\in \cf(M;S^+M)$ and
\beq\label{ttoi}
\lim_{t\to+\infty} \one_{\mp \clopen{0,+\infty}}(A(t))u(t)=0,
\eeq
Then $\wf(u)\subset \cN^\pm$. Furthermore, the analogue for  $t\to-\infty$ holds true.
\end{proposition}
\begin{proof} For the sake of definiteness  we focus on the `$+$' case in \eqref{ttoi}. Let $\widetilde u = U u$. Then, by \eqref{ttoi}, the asymptotic datum $\widetilde\varrho_+ \widetilde u=\lim_{t\to+\infty} e^{-i t H(t)}u(t)$ satisfies
\beq\label{tcbw}
\one_{\clopen{0,+\infty}}(H_+)  \widetilde\varrho_+ \widetilde u=0.
\eeq
Furthermore, $\widetilde u$ solves $\widetilde{D}\widetilde{u}= \widetilde f$, where $\widetilde f = c\, U f$. Therefore,  we can express it in terms of the asymptotic data using formula \eqref{eq:rec}, which gives
\beq\label{errrd}
\widetilde u(t)=\lim_{s\to +\infty} U(t,s) e^{isH(s)} \widetilde\varrho_+ \widetilde u - \int_t^{+\infty}U(t,s)\widetilde f (s) ds.
\eeq
The second term is smooth and will not contribute to $\wf(u)$, so we can assume without loss of generality that $\widetilde f=0$.
Let now $\chi\in S^0(\rr,[0,1])$ be equal $0$ on $\opencl{-\infty,\frac{1}{4}}$  and $1$ on $\clopen{\12,+\infty}$.
We now apply Proposition \ref{lemkey} in the same way as in the proof of Lemma \ref{compactoperator}. This gives us the existence of $R(t)\in\pdos{-1}{0}{\rr}$ such that
$$
 \big(\chi(H(t)) +  R(t)\big) U(t,s) = U(t,s)\big( \chi(H(s)) + R(s) \big)
$$
and
\beq
\slim_{s\to\infty} U(0,s) R(s) U(s,0)\in \Psi^{-\infty}(S\Sigma).
\eeq
In consequence, in combination with \eqref{errrd} we obtain
\beq\label{kjkjkji}
\bea
 \big(\chi(H(t)) +  R(t)\big) \tilde{u}(t)&=  \big(\chi(H(t)) +  R(t)\big)\lim_{s\to +\infty} U(t,s) e^{isH(s)}\widetilde \varrho_+ \widetilde u \\
 &= \lim_{s\to +\infty} U(t,s)\big( \chi(H(s)) + R(s) \big)e^{isH(s)} \widetilde\varrho_+ \widetilde u\\
 &= \lim_{s\to +\infty} U(t,s) e^{isH(s)}  \chi(H(s)) \widetilde\varrho_+ \widetilde u \fantom +U(t,0)  \lim_{s\to +\infty}  (U(0,s)  R(s) U(s,0)) W(0,s) \widetilde\varrho_+ \widetilde u \\ & =U(t,0)  \lim_{s\to +\infty}  (U(0,s)  R(s) U(s,0)) W(0,s) \widetilde\varrho_+ \widetilde u \fantom \in \cf(\rr\times \Sigma;S\Sigma).
 \eea
\eeq
Note that if we interpret $\one\otimes  \big(\chi(H(t)) +  R(t)\big)$ as an operator acting jointly on spatio-temporal variables, it is not pseudo-differential (tensor products of pseudo-differential operators are not necessarily pseudo-differential). However, an argument from \cite[78, Thm.~18.1.35]{HormanderIII} yields a slight modification which is pseudo-differential. More precisely, for all $q\in \cN^+$  we can find $B_0\in \Psi^0(\rr\times \Sigma)$ such that
$
B\defeq  U^{-1}\circ B_0 \circ (\one \otimes (\chi(H) + R)) \circ U
$
is a pseudo-differential operator in  $\Psi^0(M;S^+M)$ and is elliptic at $q$.  Since by \eqref{kjkjkji}, $B u$ is smooth,  we conclude that $q\notin \wf(u)$. Therefore, $\wf(u) \subset \cN^-$ as claimed.
\end{proof}

From the definition of $\cX_\APS^m$ and Proposition  \ref{propwf} it follows immediately that if $u\in \cX_\APS$ and $Du=0$ then $\wf(u)\subset \cN^+ \cap \cN^-=\emptyset$.

\begin{corollary}\label{coroc} If $u\in \Ker D_{\APS}$ then $u\in \cf(M;S^+M)$.
\end{corollary}

Let now $D_\APS^{\inv} :  \cY^m\to \cX_\APS^m$   be a Fredholm inverse of $D_\APS:\cX_\APS^m\to \cY^m$ associated to a complement $\cK^m$ of $\Ker D_{\APS}$ in $\cX^m$ and to a complement $\cR^m$ of $\Ran D_{\APS}$ in $\cY^m$. We can choose the complement in a compatible way, in the sense that
\beq\label{compatible}
\cK^{m_2}\subset \cK^{m_1}, \quad    \cR^{m_2}\subset \cR^{m_1}
\eeq
 for $m_2\geq m_1$. Furthermore, by density of $\cf(M;S^+M)$ in $\cY^m$ we can choose the finite dimensional space $\cR^m$ in such way that $\cR^m\subset \cf(M;S^+M)$.

Then, $D_\APS^\inv$ is a parametrix in the sense that
\beq\label{ptl}
D\circ D_\APS^\inv = \one + L, \quad D_\APS^\inv\circ D = \one + R
\eeq
acting on, say, $H_{\rm c}^m(M;S^+M)\subset \cY^m\cap \cX^m_{\APS}$,  where $L,R :  H_{\rm c}^m(M;S^+M) \to \cf(M;S^+M)$ are smoothing and of finite rank. The $m$-independent  notation is justified by the compatibility inclusions \eqref{compatible}.

\begin{theorem}\label{thm:final2} Under  the assumptions of Theorem \ref{thm:final}, let $D_\APS^\inv : \cY^m \to \cX^m_\APS$ be a Fredholm inverse of $D_\APS$ satisfying   \eqref{ptl}. Then $D_\APS^\inv$ has Feynman wavefront set.
\end{theorem}
\begin{proof} Let $D_\pm^{-1}:H^m_\c(M;S^+M) \to H^m_\loc(M;S^+M) $ be the retarded/advanced inverse  of $D$, see e.g.~\cite[\S4]{Islam2020} and references therein for microlocal properties of $D_\pm^{-1}$. Then,
$$
D(D_\APS^{\inv} - D_\pm^{-1} )= L, \quad (D_\APS^{\inv} - D_\pm^{-1} )D= R.
$$
Furthermore, for all $f\in H^m_\c(M;S^+M)$, $u=(D_\APS^{\inv} - D_\pm^{-1} )f$ satisfies
$$
\lim_{t\to \mp\infty} \one_{\pm\clopen{0,+\infty}}\big(A(t)\big)u(t)=0
$$
because
$
D_{\APS}^\inv f \in \cX_\APS^m
$
and by definition, $D_\pm^{-1}$ is supported in the causal future/past of $\supp f$. Therefore, if $B_\pm\in \Psi^0(M;S^+M)$ is constructed as the operator $B$ in the proof of Proposition \ref{propwf} then by the computation therein we find that $B_\pm(D_\APS^{\inv}-D_\pm^{-1}):H_\c^m(M;S^+M) \to \cf(M;S^+M)$ continuously for all $m\in\rr$. Therefore,
\beq\label{rrtt}
\wf'(D_\APS^{\inv}-D_\pm^{-1})\subset (\cN^\mp \cup \zero)\times T^*M.
\eeq
By the relationships \eqref{eq:equivalence2}--\eqref{eq:daps} between $D$, $D_{\rm APS}$ and the operators  $\widetilde{D}$, $\widetilde{D}_{\rm APS}$ from \secs{sec:scattering}{sec:fredholm}, there exists a Fredholm inverse $\widetilde{D}_{\rm APS}^\inv$ of $\widetilde{D}_{\rm APS}$ such that
$$
{D}_\APS^{\inv}-{D}_\pm^{-1}=  U^{-1} (\widetilde{D}_{\rm APS}^\inv- \widetilde{D}_\pm^{-1}) c U .
$$
By Proposition \ref{prop:positivity}, $\widetilde{D}_{\rm APS}^\inv- \widetilde{D}_\pm^{-1}$ is  a positive operator modulo smooth terms. This allows us to employ a standard Cauchy--Schwarz inequality argument, see e.g.~\cite[Lem.~8.4.6]{G}, to conclude that  $\wf'(\widetilde{D}_\APS^{\inv}-\widetilde{D}_\pm^{-1})$ is symmetric. Hence   $\wf'({D}_\APS^{\inv}-{D}_\pm^{-1})$ is symmetric, which in combination with \eqref{rrtt} yields
$$
\wf'(D_\APS^{\inv}-D_\pm^{-1})\subset (\cN^\mp \cup \zero)\times(\cN^\mp \cup \zero).
$$
In particular, the two wavefront sets corresponding to different signs are disjoint. In view of the identity
$$
(D_\APS^{\inv}-D_+^{-1} )- (D_\APS^{\inv}-D_-^{-1} ) = D_-^{-1}- D_+^{-1},
$$
this implies
\beq\label{pkpkpk}
\wf'(D_\APS^{\inv}-D_\pm^{-1})= \wf'(D_-^{-1}- D_+^{-1})\cap (\cN^\mp \times \cN^\mp).
\eeq
It is well known that
$$
\bea
\wf'(D_\pm^{-1})\setminus {T^*_\Delta (M\times M)} &\subset \{ (q_1, q_2) \in  \cN^\pm \times \cN^\pm \st q_1 \succ q_2  \} \fantom \cup  \{ (q_1, q_2) \in  \cN^\mp \times \cN^\mp \st q_2 \succ q_1  \}
\eea
$$
and $\wf'( D_-^{-1}- D_+^{-1} )\subset \{ (q_1, q_2) \in  \cN \times \cN \st q_1 \succ q_2  \mbox{ or } q_2\succ q_1  \}$, see e.g.~\cite[\S4]{Islam2020}, cf.~Remark \ref{remrem} for a convenient rephrasing  of the relation $q_1\succ q_2$.
From the identity 
$$
D_\APS^{\inv}=(D_\APS^{\inv}-D_\pm^{-1})+D_\pm^{-1}
$$ combined with \eqref{pkpkpk}, we can deduce that
\begin{align*}
\wf'(D_\APS^{\inv})\setminus  {T^*_\Delta (M\times M)}&\subset\textstyle\bigcap_\pm \big(\wf'(D_\APS^{\inv}-D_\pm^{-1})\cup \wf'(D_\pm^{-1})\big)\setminus  {T^*_\Delta (M\times M)}\\
&\subset  \textstyle\bigcap_\pm \big(\{ (q_1, q_2) \in  \cN^{\mp} \times \cN^{\mp} \st q_1 \succ q_2  \mbox{ or } q_2\succ q_1  \} \\ & \phantom{\subset\textstyle\bigcap_\pm i } \cup \{ (q_1, q_2) \in  \cN^\pm \times \cN^\pm \st q_1 \succ q_2  \}\big) \\
&\subset  \{ (q_1, q_2) \in  \cN^+ \times \cN^+ \st q_1 \succ q_2  \}\\ & \phantom{\subset\,}  \cup  \{ (q_1, q_2) \in  \cN^- \times \cN^- \st q_1 \succ q_2  \}\\
&\subset  \{ (q_1, q_2) \in  \cN\times \cN \st q_1 \succ q_2  \}
 \end{align*}
which implies the bound \eqref{tobef} on the wavefront set $\wf'(D_\APS^{\inv})$.
\end{proof}




\appendix

 \section{}\label{app}\init

\subsection{Abstract Fredholm theory}\label{app:fredholm}

Our objective in this subsection is to provide an abstract version of the results in \cite[\S2--3]{BS} and complement them with a useful formula for Fredholm inverses.

If $A$, $B$ are two operators, we will say that they \emph{have equal index} either if both are Fredholm and $\ind A=\ind B$, or if neither  is Fredholm (in which case they  are said to have ``infinite index'').

Following \cite{BS}, we start by recalling  the following elementary lemma, which can be found in, e.g., \cite[Prop.~A.1]{BB}.

\begin{lemma}\label{lem:fredholm} Let $\cX$ be a Hilbert space and $\cE$, $\cF$ Banach spaces. Let  $K:\cX\to\cE$, $L:\cX\to\cF$ be bounded and assume that $L$ is surjective. Then $K:\Ker L \to \cE$ and $K\oplus L:\cX\to\cE\oplus \cF$ have equal index.
\end{lemma}
\proof Since $\cX$ is Hilbert, there exists a closed subspace $\cK\subset\cX$ complementary to $\Ker L$. A direct computation shows that
\beq\label{eq:eee}
L \oplus K = \begin{pmatrix} L & 0 \\ K|_{\cK} & K|_{\Ker L}\end{pmatrix} = \begin{pmatrix} 1 & 0 \\ KL^{-1}& 1  \end{pmatrix}  \begin{pmatrix} 1 & 0  \\  0 & K|_{\Ker L}  \end{pmatrix}
 \begin{pmatrix} L & 0  \\ 0 & 1 \end{pmatrix},
\eeq
where the matrix notation refers to the decomposition $\cX= \cK \oplus \Ker L $. In view of the identity
$$
\begin{pmatrix} 1 & 0 \\ KL^{-1} & 1 \end{pmatrix} \begin{pmatrix} 1 &  0\\ -KL^{-1} & 1  \end{pmatrix} = \begin{pmatrix} 1 & 0  \\  0 & 1  \end{pmatrix},
$$
the first matrix on the r.h.s.~of  \eqref{eq:eee} is an isomorphism, and since the third one is also an isomorphism,  one gets from \eqref{eq:eee} that $L\oplus K$ and $\one\oplus (K|_{\Ker L})$ have equal index. One concludes by observing that $\one\oplus (K|_{\Ker L})$ and  $(K|_{\Ker L})$ have equal index.
\qeds

Let $\cH_+,\cH_-$ be two Banach spaces. Let $\pi^{\pm}_+$ be a pair of complementary projections on $\cH_+$, and $\pi^{\pm}_-$ be a pair of complementary projections on $\cH_-$.
Let us denote
$$
\cH^\pm_+=\pi^\pm_+\cH_+, \  \ \cH^\pm_-=\pi^\pm_-\cH_-.
$$
If $W: \cH_-\to\cH_+$, we use the matrix notation
\beq\label{mnot}
W= \begin{pmatrix}W^{++} & W^{+-} \\ W^{-+} & W^{--}   \end{pmatrix}: \cH^+_- \oplus \cH_-^- \to  \cH_+^+\oplus \cH_+^-,
\eeq
where $W^{++}=\pi_+^+ W \imath_-^+: \cH_-^+\to \cH_+^+$ with $\imath_-^+: \cH^+_-\to \cH_-$ the canonical injection, and similarly for the other components.
We state below an abstract version of \cite[Thm.~3.2]{BS} combined with \cite[Thm.~3.3]{BS}. The proof  is largely analogous to \cite{BS}.

\begin{proposition} \label{prop:fredholm} Let $\cX$ be a Hilbert space and $\cY$ a Banach space. Let $P:\cX\to \cY$ and  $\varrho_\pm:\cX\to\cH_\pm$ be bounded, and suppose that:
\beq\label{eq:fr1}
\varrho_\pm \oplus  P : \cX\to\cH_\pm\oplus \cY \mbox{ is boundedly invertible}.
\eeq
Let $\varrho=\pi_-^+\varrho_- \oplus \pi_+^-\varrho_+$. Define\footnote{Equivalently, $W$ can be expressed by the formula $Wh= \varrho_+ \circ (\varrho_- \oplus P)^{-1} (h,0)$
for all $h\in\cH_-$.}
$$
W=\varrho_+\varrho_-^{-1}: \cH_-\to \cH_+,
$$
where $\varrho_-^{-1}:\cH_-\to \Ker P$ is the inverse of $\varrho_-|_{\Ker P}$. If $W^{+-}$ is compact then  $W^{--}$ is Fredholm. Moreover, $W^{--}$ and $P: {\Ker\varrho}\to \cY$
have equal index.
\end{proposition}
\proof \step{1} Observe that $W$ is invertible of inverse $W^{-1}=\varrho_-\varrho_+^{-1}$. Therefore, in the matrix notation generalizing  \eqref{mnot}, the identity $W^{-1} W= \one$
implies $(W^{-1})^{--} \circ W^{--}+ (W^{-1})^{-+} \circ W^{+-}=\one$, so $(W^{-1})^{--} \circ W^{--}$ equals $\one$ modulo a compact term, hence $W^{--}$ is Fredholm.

\step{2} We claim that the operators $W^{--}:\cH_-^-\to \cH_+^-$  and
\beq\label{eq:gk0}
\varrho :  \Ker P \to \cH_-^+ \oplus \cH_+^-
\eeq
have equal index. On $\Ker P$, we have
$$
\varrho =  \pi_-^+\varrho_- \oplus \pi_+^-\varrho_+ =  (\pi_-^+ \oplus   \pi_+^- W)\varrho_-,
$$
where $\varrho_-: \Ker P \to \cH_-=\cH_-^+  \oplus \cH_-^-$ is an isomophism by \eqref{eq:fr1}. Furthermore,
\begin{align*}
(\pi_-^+ \oplus   \pi_+^- W):  \cH_-^+  \oplus \cH_-^-  &\to \cH_-^+\oplus \cH_+^- \\
(u^+,u^-)& \mapsto (u^+,W^{-+}u^+ + W^{--}u^-)
\end{align*}
 is represented by the matrix
 \beq\label{eq:gk}
 \begin{pmatrix} \one & 0 \\  W^{-+} &  W^{--} \end{pmatrix}  =
  \begin{pmatrix} \one & 0 \\  W^{-+} &  \one  \end{pmatrix}    \begin{pmatrix} \one & 0 \\  0 &  W^{--} \end{pmatrix} ,
 \eeq
 where $\left(\begin{smallmatrix} \one & 0 \\  W^{-+} &  \one  \end{smallmatrix}\right)$ is an isomorphism  with inverse  $\left(\begin{smallmatrix} \one & 0 \\  -W^{-+} &  \one  \end{smallmatrix}\right)$. Observe that $W^{--}$ and $\left(\begin{smallmatrix} \one & 0 \\  0 &  W^{--} \end{smallmatrix}\right)$ have equal index. Up to composition with isomorphisms, the latter operator coincides with \eqref{eq:gk}, and thus with \eqref{eq:gk0}. Therefore, $W^{--}$ and \eqref{eq:gk0} have equal index as claimed.

Next, by Lemma \ref{lem:fredholm} applied to $\cE=\Ran \varrho= \cH_-^+ \oplus \cH_+^- $, $\cF=\cY$, $K=\varrho$ and $L=P$,  the operators
\beq\label{eq:fr4}
\varrho\oplus  P : \cX\to \cH_-^+ \oplus \cH_+^-  \oplus \cY
\eeq
and \eqref{eq:gk0} have equal index.
Finally, by Lemma \ref{lem:fredholm}   applied to  $\cE=\cY$, $\cF=\cH_-^+ \oplus \cH_+^-$,  $K=P$ and $L=\varrho$, the operators \eqref{eq:fr4} and $P: {\Ker\varrho}\to \cY$ have equal index.
We conclude that $W^{--}$ and $P|_{\Ker \varrho} $ have equal index.
\qeds

\begin{proposition}\label{propq} Under the same assumptions as in Proposition \ref{prop:fredholm},  supposing in addition  that $W^{-+}$ is compact, let us define $P_\mp^{-1}= (\varrho_\pm \oplus P)^{-1}\circ (0\oplus \one):\cY\to \cX$, and let
$$
Q= (\one  - \varrho_-^{-1} \pi_-^+ \varrho_-)P_-^{-1}:\cY\to \cX.
$$
Then $Q$ satisfies $P\circ Q=\one$  on $\cY$ and
\beq\label{roq}
\quad \varrho\circ Q : \cY \to \cHH \mbox{ is compact}.
\eeq
Furthermore, if $(P|_{\Ker\varrho})^{\inv}$ is a Fredholm inverse of $P:\Ker\varrho\to \cY$, then
\beq\label{pkq}
(P|_{\Ker\varrho})^{\inv} - Q : \cY \to \cX \mbox{ is compact}.
\eeq
\end{proposition}
\proof \step{1} The property $P\circ Q=\one$ follows from $P\circ P_-^{-1}=\one$ and the fact that by definition, $\varrho_-^{-1}$ maps to $\Ker P$.

Let us show the compactness of  $\varrho\circ Q$. On the one hand, using that $\varrho =  \pi_-^+\varrho_- \oplus \pi_+^-\varrho_+ =  (\pi_-^+ \oplus   \pi_+^- W)\varrho_-$ on $\Ker P$, we can write
\beq\label{roro1}
\bea\varrho\, \varrho_-^{-1} \pi_-^+ \varrho_- & =  (\pi_-^+ \oplus   \pi_+^- W)\pi_-^+ \varrho_-\\
& =   (\pi_-^+ \oplus   \pi_+^- W \pi^+_-) \varrho_- =  \pi_-^+ \varrho_- \oplus W^{-+} \pi^+_-\varrho_-
\eea
\eeq
on $\Ker P$. On the other hand,
\beq\label{roro2}
\varrho=  \pi_-^+\varrho_- \oplus \pi_+^-\varrho_+  =   \pi_-^+\varrho_- \oplus  0 \ \mbox{ on } \Ker \varrho_+.
\eeq
Since $P_\mp^{-1}$ maps to  $\Ker \varrho_\pm$, using \eqref{roro1} and \eqref{roro2} we get
\beq\label{K1}
\bea
\varrho\circ Q&= \varrho(\one - \varrho_-^{-1} \pi^+_-  \varrho_-) P_-^{-1}\\ &= \varrho  P_-^{-1} -  \varrho\, \varrho_-^{-1} \pi_-^+ \varrho_- P_-^{-1}\\
&= \varrho  P_-^{-1} -  \varrho\, \varrho_-^{-1} \pi_-^+ \varrho_- (P_-^{-1}-P_+^{-1}) \\
&= (\pi_-^+\varrho_- \oplus  0) P_-^{-1} -  \varrho\, \varrho_-^{-1} \pi_-^+ \varrho_- (P_-^{-1}-P_+^{-1})
\\
&= (\pi_-^+\varrho_- \oplus  0) (P_-^{-1}-P_+^{-1}) -  \varrho\, \varrho_-^{-1} \pi_-^+ \varrho_- (P_-^{-1}-P_+^{-1}) \\
&= (\pi_-^+\varrho_- \oplus  0) (P_-^{-1}-P_+^{-1}) -  (\pi_-^+ \varrho_- \oplus W^{-+} \pi^+_-\varrho_-)(P_-^{-1}-P_+^{-1}) \\
&= -  (0\oplus W^{-+} \pi^+_-\varrho_-)(P_-^{-1}-P_+^{-1})
 \eqdef K_1,
\eea
\eeq
which is compact by compactness of $W^{-+}$.

\step{2} We know from the proof of Proposition \ref{prop:fredholm} that $\varrho\oplus  P : \cX\to \cH_-^+ \oplus \cH_+^-  \oplus \cY$ is Fredholm. Let $(\varrho\oplus P)^{\inv}$ be a Fredholm inverse; then in particular
\beq\label{beka1}
(\varrho\oplus P)^{\inv} (\varrho\oplus P)=\one - \pi_{\ker (\varrho\oplus P)},
\eeq
where $\pi_{\ker (\varrho\oplus P)}:\cX\to \cX$ projects to the finite dimensional space $\Ker (\varrho\oplus P)$. Furthermore, $(\varrho\oplus P)^{\inv}$ can be chosen in such way that
\beq\label{beka2}
(\varrho\oplus P)^{\inv}  \circ (0\oplus \one) =  (P|_{\Ker\varrho})^{\inv}.
\eeq
Indeed, this can be arranged by defining $(\varrho\oplus P)^{\inv}$ as
$$
(\varrho\oplus P)^{\inv}= \begin{pmatrix} \varrho^{-1} & 0 \\ 0 & 1\end{pmatrix} \begin{pmatrix} 1 & 0 \\ 0 & (P|_{\Ker\varrho})^{\inv}\end{pmatrix}
 \begin{pmatrix} 1 & 0 \\ -P\varrho^{-1}   & 1\end{pmatrix},$$
 where the matrix notation refers to decomposition $\cX=\Ran(P|_{\Ker\varrho})^{\inv}\oplus\Ker\varrho$, and where $\varrho^{-1}:  \cHH\to \Ran (P|_{\Ker\varrho})^{\inv}$ is well-defined as the inverse of $\varrho$ restricted to $\Ran (P|_{\Ker\varrho})^{\inv}$. The fact that this defines a Fredholm inverse can be checked directly or by using \eqref{eq:eee} with $L=\varrho$ and $K=P$.

 If $K_1$ is the compact operator defined in \eqref{K1} then using \eqref{beka1} and \eqref{beka2} we get
\beq\label{K2}
\bea
Q & = (\varrho\oplus P)^{\inv} (\varrho\oplus P)Q +  \pi_{\ker (\varrho\oplus P)} Q\\
&= (\varrho\oplus P)^{\inv} \circ (K_1\oplus \one) +  \pi_{\ker (\varrho\oplus P)} Q \\
&=(P|_{\Ker\varrho})^{\inv} +(\varrho\oplus P)^{\inv} \circ (K_1\oplus 0)+  \pi_{\ker (\varrho\oplus P)} Q \\
&\eqdef (P|_{\Ker\varrho})^{\inv}-  K,
\eea
\eeq
where $K$ is compact by compactness of $K_1$ and $\pi_{\ker (\varrho\oplus P)}$.
\qed

\bibliographystyle{abbrv}
\bibliography{diracfredholm}

\begin{thebibliography}{10}

\bibitem{Assal2016}
M.~Assal.
\newblock {Long time semiclassical Egorov theorem for $h$-pseudodifferential
  systems}.
\newblock {\em Asymptot. Anal.}, 101(1-2):17--67, 2016.

\bibitem{Atiyah1975a}
M.~F. Atiyah, V.~K. Patodi, and I.~M. Singer.
\newblock {Spectral asymmetry and Riemannian geometry. I}.
\newblock {\em Math. Proc. Cambridge Philos. Soc.}, 77(1):43--69, 1975.

\bibitem{BB}
W.~Ballmann and C.~B{\"{a}}r.
\newblock {Boundary value problems for elliptic differential operators of first
  order}.
\newblock {\em Surv. Differ. Geom.}, 17(1):1--78, 2012.

\bibitem{BGM}
C.~B{\"{a}}r, P.~Gauduchon, and A.~Moroianu.
\newblock {Generalized cylinders in semi-Riemannian and spin geometry}.
\newblock {\em Math. Zeitschrift}, 249(3):545--580, 2005.

\bibitem{Bar2018}
C.~B{\"{a}}r and S.~Hannes.
\newblock {Boundary value problems for the Lorentzian Dirac operator}.
\newblock In {\em Geom. Phys. Vol. I}, pages 3--18. Oxford University Press,
  oct 2018.

\bibitem{BS2}
C.~B{\"{a}}r and A.~Strohmaier.
\newblock {A rigorous geometric derivation of the chiral anomaly in curved
  backgrounds}.
\newblock {\em Commun. Math. Phys.}, 347(3):703--721, 2016.

\bibitem{BS}
C.~B{\"{a}}r and A.~Strohmaier.
\newblock {An index theorem for Lorentzian manifolds with compact spacelike
  Cauchy boundary}.
\newblock {\em Am. J. Math.}, 141(5):1421--1455, 2019.

\bibitem{Baer2020a}
C.~B{\"{a}}r and A.~Strohmaier.
\newblock {Local index theory for Lorentzian manifolds}.
\newblock {\em arXiv:2012.01364}, 2020.

\bibitem{BVW}
D.~Baskin, A.~Vasy, and J.~Wunsch.
\newblock {Asymptotics of radiation fields in asymptotically Minkowski space}.
\newblock {\em Am. J. Math.}, 137(5):1293--1364, 2015.

\bibitem{Baum1981}
H.~Baum.
\newblock {Spin-Strukturen und Dirac-Operatoren {\"{u}}ber pseudo-Riemannschen
  Mannigfaltigkeiten}.
\newblock In Teubner, editor, {\em Teubner-Texte zur Math.}, chapter vol. 4.
  Leipzig, 1981.

\bibitem{bernal1}
A.~N. Bernal and M.~S{\'{a}}nchez.
\newblock {On smooth Cauchy hypersurfaces and Geroch's splitting theorem}.
\newblock {\em Commun. Math. Phys.}, 243(3):461--470, 2003.

\bibitem{bernal2}
A.~N. Bernal and M.~S{\'{a}}nchez.
\newblock {Smoothness of time functions and the metric splitting of globally
  hyperbolic spacetimes}.
\newblock {\em Commun. Math. Phys.}, 257(1):43--50, 2005.

\bibitem{Bismut1990}
J.-M. Bismut and J.~Cheeger.
\newblock {Families index for manifolds with boundary, superconnections, and
  cones. I. Families of manifolds with boundary and Dirac operators}.
\newblock {\em J. Funct. Anal.}, 89(2):313--363, 1990.

\bibitem{Bismut1990a}
J.-M. Bismut and J.~Cheeger.
\newblock {Families index for manifolds with boundary, superconnections and
  cones. II. The Chern character}.
\newblock {\em J. Funct. Anal.}, 90(2):306--354, 1990.

\bibitem{Bolte2004}
J.~Bolte and R.~Glaser.
\newblock {A semiclassical Egorov theorem and quantum ergodicity for matrix
  valued operators}.
\newblock {\em Commun. Math. Phys.}, 247(2):391--419, 2004.

\bibitem{Braverman2020}
M.~Braverman.
\newblock {An index of strongly Callias operators on Lorentzian manifolds with
  non-compact boundary}.
\newblock {\em Math. Zeitschrift}, 294(1-2):229--250, 2020.

\bibitem{Braverman2021}
M.~Braverman and P.~Shi.
\newblock {The Atiyah–Patodi–Singer index on manifolds with non-compact
  boundary}.
\newblock {\em J. Geom. Anal.}, 31(4):3713--3763, 2021.

\bibitem{Brummelhuis1999}
R.~Brummelhuis and J.~Nourrigat.
\newblock {Scattering amplitude for dirac operators}.
\newblock {\em Commun. Partial Differ. Equations}, 24(1-2):377--394, 1999.

\bibitem{Bunke1992}
U.~Bunke and T.~Hirschmann.
\newblock {The index of the scattering operator on the positive spectral
  subspace}.
\newblock {\em Commun. Math. Phys.}, 148(3):487--502, 1992.

\bibitem{Capoferri2021}
M.~Capoferri and D.~Vassiliev.
\newblock {Invariant subspaces of elliptic systems I: pseudodifferential
  projections}.
\newblock {\em arXiv:2103.14325}, 2021.

\bibitem{Cordes1982}
H.~Cordes.
\newblock {A version of Egorov's theorem for systems of hyperbolic
  pseudo-differential equations}.
\newblock {\em J. Funct. Anal.}, 48(3):285--300, 1982.

\bibitem{cordes}
H.~O. Cordes.
\newblock {\em {Precisely Predictable Dirac Observables}}, volume 154 of {\em
  Fundamental Theories of Physics}.
\newblock Springer Netherlands, Dordrecht, 2007.

\bibitem{Damaschke2021}
O.~Damaschke.
\newblock {Atiyah-Singer Dirac operator on spacetimes with non-compact Cauchy
  hypersurface}.
\newblock {\em arXiv:2107.08532}, 2021.

\bibitem{Dang2020a}
N.~V. Dang and M.~Wrochna.
\newblock {Complex powers of the wave operator and the spectral action on
  Lorentzian scattering spaces}.
\newblock {\em arXiv:2012.00712}, 2020.

\bibitem{Dappiaggi2009}
C.~Dappiaggi, V.~Moretti, and N.~Pinamonti.
\newblock {Distinguished quantum states in a class of cosmological spacetimes
  and their Hadamard property}.
\newblock {\em J. Math. Phys.}, 50(6):062304, 2009.

\bibitem{davies}
E.~B. Davies.
\newblock {The functional calculus}.
\newblock {\em J. London Math. Soc.}, 52(1):166--176, 1995.

\bibitem{derezinskisiemssen0}
J.~Derezi{\'{n}}ski and D.~Siemssen.
\newblock {Feynman propagators on static spacetimes}.
\newblock {\em Rev. Math. Phys.}, 30(03):1850006, 2018.

\bibitem{derezinskisiemssen2}
J.~Derezi{\'{n}}ski and D.~Siemssen.
\newblock {An evolution equation approach to linear Quantum Field Theory}.
\newblock {\em arXiv:1912.10692}, 2019.

\bibitem{Drago2021}
N.~Drago, N.~Gro{\ss}e, and S.~Murro.
\newblock {The Cauchy problem of the Lorentzian Dirac operator with APS
  boundary conditions}.
\newblock {\em arXiv:2104.00585}, 2021.

\bibitem{Drouot2019}
A.~Drouot.
\newblock {Characterization of edge states in perturbed honeycomb structures}.
\newblock {\em Pure Appl. Anal.}, 1(3):385--445, 2019.

\bibitem{DH}
J.~J. Duistermaat and L.~H{\"{o}}rmander.
\newblock {Fourier integral operators. II}.
\newblock {\em Acta Math.}, 128:183--269, 1972.

\bibitem{Finster2017}
F.~Finster.
\newblock {The chiral index of the fermionic signature operator}.
\newblock {\em Math. Res. Lett.}, 24(1):37--66, 2017.

\bibitem{GHV}
J.~Gell-Redman, N.~Haber, and A.~Vasy.
\newblock {The Feynman propagator on perturbations of Minkowski space}.
\newblock {\em Commun. Math. Phys.}, 342(1):333--384, 2016.

\bibitem{G}
C.~G{\'{e}}rard.
\newblock {\em {Microlocal Analysis of Quantum Fields on Curved Spacetimes}}.
\newblock European Mathematical Society, Z{\"{u}}rich, 2019.

\bibitem{Gerard2021a}
C.~G{\'{e}}rard and T.~Stoskopf.
\newblock {Hadamard property of the in and out states for Dirac fields on
  asymptotically static spacetimes}.
\newblock {\em arXiv:2108.11955}, 2021.

\bibitem{GS}
C.~G{\'{e}}rard and T.~Stoskopf.
\newblock {Hadamard states for quantized Dirac fields on Lorentzian manifolds
  of bounded geometry}.
\newblock {\em arXiv:2108.11630}, 2021.

\bibitem{Gerard2020}
C.~G{\'{e}}rard and M.~Wrochna.
\newblock {The Feynman problem for the Klein--Gordon equation}.
\newblock {\em S{\'{e}}minaire Laurent Schwartz --- EDP Appl.},
  (2019-2020):Expos{\'{e}} no IV.

\bibitem{inout}
C.~G{\'{e}}rard and M.~Wrochna.
\newblock {Hadamard property of the \emph{in} and \emph{out} states for
  Klein–Gordon fields on asymptotically static spacetimes}.
\newblock {\em Ann. Henri Poincar{\'{e}}}, 18(8):2715--2756, 2017.

\bibitem{massivefeynman1}
C.~G{\'{e}}rard and M.~Wrochna.
\newblock {The massive Feynman propagator on asymptotically Minkowski
  spacetimes}.
\newblock {\em Am. J. Math.}, 141(6):1501--1546, 2019.

\bibitem{Gerard2018}
C.~G{\'{e}}rard and M.~Wrochna.
\newblock {The massive Feynman propagator on asymptotically Minkowski
  spacetimes II}.
\newblock {\em Int. Math. Res. Not.}, 2020(20):6856--6870, 2020.

\bibitem{Gilkey1993}
P.~Gilkey.
\newblock {On the index of geometrical operators for Riemannian-manifolds with
  boundary}.
\newblock {\em Adv. Math. (N. Y).}, 102(2):129--183, 1993.

\bibitem{sylvain}
S.~Gol{\'{e}}nia and T.~Jecko.
\newblock {Weighted Mourre's commutator theory, application to
  Schr{\"{o}}dinger operators with oscillating potential}.
\newblock {\em J. Oper. Theory}, 70(1):109--144, 2103.

\bibitem{Gerd1992}
G.~Grubb.
\newblock {Heat operator trace expansions and index for general
  Atiyah–Patodi–Singer boundary problems}.
\newblock {\em Commun. Partial Differ. Equations}, 17(11-12):2031--2077, 1992.

\bibitem{Hintz2017}
P.~Hintz.
\newblock {Resonance expansions for tensor-valued waves on asymptotically
  Kerr–de Sitter spaces}.
\newblock {\em J. Spectr. Theory}, 7(2):519--557, 2017.

\bibitem{H}
L.~H{\"{o}}rmander.
\newblock {\em {The Analysis of Linear Partial Differential Operators I.
  Distribution Theory and Fourier Analysis}}.
\newblock Springer Verlag, Berlin, second edition, 1990.

\bibitem{HormanderIII}
L.~H{\"{o}}rmander.
\newblock {\em {The Analysis of Linear Partial Differential Operators III.
  Pseudo-Differential Operators}}.
\newblock Classics in Mathematics. Springer Berlin Heidelberg, Berlin,
  Heidelberg, 2007.

\bibitem{Islam2020}
O.~Islam and A.~Strohmaier.
\newblock {On microlocalization and the construction of Feynman propagators for
  normally hyperbolic operators}.
\newblock {\em arXiv:2012.09767}, 2020.

\bibitem{Jakobson2007}
D.~Jakobson and A.~Strohmaier.
\newblock {High energy limits of Laplace-type and Dirac-type eigenfunctions and
  frame flows}.
\newblock {\em Commun. Math. Phys.}, 270(3):813--833, 2007.

\bibitem{Kordyukov2007}
Y.~A. Kordyukov.
\newblock {The Egorov theorem for transverse Dirac-type operators on foliated
  manifolds}.
\newblock {\em J. Geom. Phys.}, 57(11):2345--2364, 2007.

\bibitem{Lawson1989}
H.~Lawson and M.-L. Michelsohn.
\newblock {\em {Spin Geometry}}.
\newblock Princeton University Press, Princeton, 1989.

\bibitem{Lesch2004}
M.~Lesch.
\newblock {The uniqueness of the spectral flow on spaces of unbounded
  self-adjoint Fredholm operators}.
\newblock In {\em Spectr. Geom. Manifolds with Bound. (Editors B.
  Boo{\ss}--Bavnbek, G. Grubb, K.P. Wojciechowski)}. American Mathematical
  Society, Providence, 2004.

\bibitem{lesch}
M.~Lesch.
\newblock {\em {Spectral geometry of manifolds with boundary and decomposition
  of manifolds}}, volume 366 of {\em Contemporary Mathematics}.
\newblock American Mathematical Society, Providence, Rhode Island, 2005.

\bibitem{Matsui1987}
T.~Matsui.
\newblock {The index of scattering operators of Dirac equations}.
\newblock {\em Commun. Math. Phys.}, 110(4):553--571, 1987.

\bibitem{Matsui1990}
T.~Matsui.
\newblock {The index of scattering operators of Dirac equations, II}.
\newblock {\em J. Funct. Anal.}, 94(1):93--109, 1990.

\bibitem{Melrose1993}
R.~Melrose.
\newblock {\em {The Atiyah--Patodi--Singer Index Theorem}}.
\newblock CRC Press, Boca Raton, FL, 1993.

\bibitem{Melrose1997}
R.~B. Melrose and P.~Piazza.
\newblock {Families of Dirac operators, boundaries and the $b$-calculus}.
\newblock {\em J. Differ. Geom.}, 46(1), 1997.

\bibitem{Moretti2008}
V.~Moretti.
\newblock {Quantum out-states holographically induced by asymptotic flatness:
  invariance under spacetime symmetries, energy positivity and Hadamard
  property}.
\newblock {\em Commun. Math. Phys.}, 279(1):31--75, 2008.

\bibitem{nakamurataira}
S.~Nakamura and K.~Taira.
\newblock {Essential self-adjointness of real principal type operators}.
\newblock {\em Ann. Henri Lebesgue}, 4:1035--1059, 2021.

\bibitem{Pankrashkin2014}
K.~Pankrashkin and S.~Richard.
\newblock {One-dimensional Dirac operators with zero-range interactions:
  Spectral, scattering, and topological results}.
\newblock {\em J. Math. Phys.}, 55(6):062305, 2014.

\bibitem{pazy}
A.~Pazy.
\newblock {\em {Semigroups of linear operators and applications to partial
  differential equations}}, volume~44 of {\em Applied Mathematical Sciences}.
\newblock Springer New York, New York, NY, 1983.

\bibitem{phi}
J.~Phillips.
\newblock {Self-adjoint Fredholm operators and spectral flow}.
\newblock {\em Can. Math. Bull.}, 39(4):460--467, 1996.

\bibitem{RS}
M.~Reed and B.~Simon.
\newblock {\em {Methods of Modern Mathematical Physics, vol. II: Fourier
  Analysis, Self-Adjointness}}.
\newblock Academic Press, New York.

\bibitem{ronge}
L.~Ronge.
\newblock {Index theory for globally hyperbolic spacetimes}.
\newblock {\em arXiv:1910.10452}, 2019.

\bibitem{Ruzhansky2010}
M.~Ruzhansky and V.~Turunen.
\newblock {\em {Pseudo-Differential Operators and Symmetries}}.
\newblock Birkh{\"{a}}user Basel, Basel, 2010.

\bibitem{seeley}
R.~T. Seeley.
\newblock {Complex powers of an elliptic operator}.
\newblock {\em Proc. Symp. Pure Math.}, 10:288--307, 1967.

\bibitem{shubin}
M.~A. Shubin.
\newblock {\em {Pseudodifferential Operators and Spectral Theory}}.
\newblock 2001.

\bibitem{Sjostrand1993}
J.~Sj{\"{o}}strand.
\newblock {Projecteurs adiabatiques du point de vue pseudodiff{\'{e}}rentiel}.
\newblock {\em C. R. Acad. Sci. Paris, S{\'{e}}r. I}, 317(2):217--220, 1993.

\bibitem{Stoskopf}
T.~Stoskopf.
\newblock {\em in preparation}.
\newblock PhD thesis, Paris-Saclay.

\bibitem{Strohmaier2020b}
A.~Strohmaier and S.~Zelditch.
\newblock {A Gutzwiller trace formula for stationary space-times}.
\newblock {\em Adv. Math.}, 376, 2020.

\bibitem{Taira2020a}
K.~Taira.
\newblock {Limiting absorption principle and equivalence of Feynman propagators
  on asymptotically Minkowski spacetimes}.
\newblock {\em Commun. Math. Phys.}, 388(1):625--655, 2021.

\bibitem{Taylor2011}
M.~E. Taylor.
\newblock {\em {Partial Differential Equations II}}, volume 116 of {\em Applied
  Mathematical Sciences}.
\newblock Springer, New York, NY, 2011.

\bibitem{Trautman2008}
A.~Trautman.
\newblock {Connections and the Dirac operator on spinor bundles}.
\newblock {\em J. Geom. Phys.}, 58(2):238--252, 2008.

\bibitem{dun1}
K.~van~den Dungen.
\newblock {Families of spectral triples and foliations of space(time)}.
\newblock {\em J. Math. Phys.}, 59(6):063507, 2018.

\bibitem{dun3}
K.~van~den Dungen and L.~Ronge.
\newblock {The APS-index and the spectral flow}.
\newblock {\em Oper. Matrices}, (4):1393--1416, 2021.

\bibitem{Vasy2013}
A.~Vasy.
\newblock {Microlocal analysis of asymptotically hyperbolic and Kerr-de Sitter
  spaces (with an appendix by Semyon Dyatlov)}.
\newblock {\em Invent. Math.}, 194(2):381--513, 2013.

\bibitem{Vasy2017b}
A.~Vasy.
\newblock {On the positivity of propagator differences}.
\newblock {\em Ann. Henri Poincar{\'{e}}}, 18(3):983--1007, 2017.

\bibitem{vasyessential}
A.~Vasy.
\newblock {Essential self-adjointness of the wave operator and the limiting
  absorption principle on Lorentzian scattering spaces}.
\newblock {\em J. Spectr. Theory}, 10(2):439--461, 2020.

\bibitem{vasywrochna}
A.~Vasy and M.~Wrochna.
\newblock {Quantum fields from global propagators on asymptotically Minkowski
  and extended de Sitter spacetimes}.
\newblock {\em Ann. Henri Poincar{\'{e}}}, 2018.

\bibitem{Zahn2015}
J.~Zahn.
\newblock {Locally covariant chiral fermions and anomalies}.
\newblock {\em Nucl. Phys. B}, 890:1--16, 2015.

\bibitem{Zworski2012}
M.~Zworski.
\newblock {\em {Semiclassical Analysis}}.
\newblock American Mathematical Society, Providence, RI, 2012.

\end{thebibliography}

 \end{document}